\documentclass[preprint]{imsart}

\RequirePackage[OT1]{fontenc}
\RequirePackage{amsthm,amsmath}
\RequirePackage[authoryear]{natbib}
\RequirePackage[colorlinks,citecolor=blue,urlcolor=blue]{hyperref}
\RequirePackage{mathrsfs}
\RequirePackage{amssymb}
\RequirePackage{stmaryrd}
\RequirePackage{bbm}
\RequirePackage{bm}
\RequirePackage{graphicx}

\startlocaldefs
\numberwithin{equation}{section}
\theoremstyle{plain}
\newtheorem{thm}{Theorem}[section]
\theoremstyle{remark}
\newtheorem{rem}{Remark}
\theoremstyle{plain}
\newtheorem{lem}{Lemma}
\theoremstyle{plain}

\theoremstyle{definition}
\newtheorem{defn}{Definition}
\theoremstyle{plain}

\endlocaldefs

\makeatletter
\newcommand{\mathleft}{\@fleqntrue\@mathmargin0pt}
\newcommand{\mathcenter}{\@fleqnfalse}
\makeatother

\begin{document}

\begin{frontmatter}
\title{High-dimensional Adaptive Minimax Sparse Estimation with Interactions}
\runtitle{Adaptive Minimax Estimation with Interactions}

\begin{aug}
  \author{\fnms{Chenglong}  \snm{Ye}\corref{}\thanksref{}\ead[label=e1]{yexxx323@umn.com}},
  \author{\fnms{Yuhong} \snm{Yang}\corref{}\thanksref{}\ead[label=e2]{yangx374@umn.edu}}%

  \runauthor{C. Ye and Y. Yang}

  \affiliation{University of Minnesota}

  \address{School of Statistics\\ University of Minnesota \\
   313 Ford Hall, 224 Church St SE\\
   Minneapolis, MN 55455, USA\\ 
          \printead{e1,e2}}

\end{aug}

\begin{abstract}
\sloppy High-dimensional linear regression with interaction effects is broadly applied in research fields such as bioinformatics and social science. In this paper, we first investigate
the minimax rate of convergence for regression estimation in high-dimensional sparse linear models with two-way
interactions. We derive matching upper and lower bounds under three
types of heredity conditions: strong heredity, weak heredity and no
heredity. From the results: (i) A stronger heredity condition may or may not drastically improve the minimax rate of convergence. In fact, in some situations, the minimax rates of convergence are the same under all three heredity conditions; (ii) The minimax rate of convergence is determined by the maximum of the total price of estimating the main effects and that of estimating the interaction effects, which goes beyond purely comparing the order of the number of non-zero main effects $r_1$ and non-zero interaction effects $r_2$; (iii) Under any of the three heredity conditions, the estimation of the interaction terms may be the dominant part in determining the rate of convergence for two different reasons: 1) there exist more interaction terms than main effect terms or 2) a large ambient dimension makes it more challenging to estimate even a small number of interaction terms. Second, we construct an adaptive estimator that achieves the minimax rate of convergence regardless of the true heredity condition and the sparsity indices $r_1, r_2$.
\end{abstract}
\begin{keyword}[class=MSC]
\kwd[Primary ]{62C20}
\kwd[; secondary ]{62J05}
\end{keyword}
\begin{keyword}
\kwd{Minimax rate of convergence}
\kwd{sparsity}
\kwd{high-dimensional regression}
\kwd{quadratic model}
\kwd{interaction selection}
\kwd{heredity condition}
\kwd{hierarchical structure}
\kwd{adaptive estimation}
\end{keyword}
\end{frontmatter}
\section{Introduction}

High-dimensional data are increasingly prevalent in various areas such as
bioinformatics, astronomy, climate science and social science. When
the number of variables $p$ is larger than the sample size $n$ in
the linear regression setting, statistical estimation of the regression
function often requires
some crucial conditions. One common condition is the sparsity of the
data generating model, under which only a small portion of the variables
are important to affect the response variable. Under this condition, both sparse estimation
of high-dimensional linear regression functions and variable selection have been well studied with fruitful theoretical
understandings in the recent decade. Minimax estimation of the regression function with main effects only are well investigated under $l_{q}$-sparsity constraints with $0\leq q\leq1$
\citep[e.g.,][]{van2007deterministic,candes2007dantzig,bunea2007aggregation,zhang2008sparsity,van2008high,van2009conditions,bickel2009simultaneous,zhang2010,knight2000asymptotics,raskutti2011minimax,Rigollet:2011fz,wang2014adaptive}; model selection consistency results are also obtained for various model selection procedures \citep[e.g.,][]{fan2001variable,zhao2006model,zhang2008sparsity,zou2008,lv2009unified}. 

However, models with only main effects are often not adequate to fully
capture the nature of the data. Interaction terms may be necessary to not only improve the prediction performance but also
enhance the understanding of the relationships among the variables,
especially in areas such as genetics, medicine and behavioristics,
where interaction effects between the covariates are of enormous
interest. Hierarchical constraints are often imposed to describe the
underlying structure of models with interaction effects, such as the
marginality principle \citep{Nelder:1977ke}, the effect heredity
principle \citep{hamada1992analysis} and the ``well-formulated models''
\citep{peixoto1987hierarchical}. We follow a popular naming convention
of heredity conditions as adopted in \citet{10.2307/3315687}: strong
heredity and weak heredity. Strong heredity assumes that if an interaction
term is in the model, then both of its corresponding main effects
should also be included, while weak heredity only requires that at
least one of its main effects should be included. In practice, it
is possible that, compared to the interaction terms, some main effects
are so small that including them in modeling may not be beneficial
from the perspective of estimation variability. Thus, in this work we 
take into consideration the additional case where no heredity condition is imposed
at all, also for the purpose of theoretical comparison with the other
two heredity conditions. 

Many approaches are proposed for interaction selection, most of which
can be categorized into two types: {\it joint selection} and {\it stage-wise
selection}. The joint selection approach selects the main and interaction
terms simultaneously by searching over all possible models with interactions.
A typical way of joint selection is to use regularization methods
with specially designed penalty terms. For example, \citet{yuan2009structured}
introduced a family of shrinkage estimators, which incorporate the
hierarchical structures through linear equality constraints on the
coefficients and possess both selection consistency and root-$n$
estimation consistency under fixed $p$. \citet{choi2010variable}
re-parameterized the regression model with interactions and applied an
adaptive $L_{1}$-norm penalty. The estimators have the oracle property \citep{fan2001variable} when $p=o(n^{1/10})$.
\citet{Hao:2017ed} proposed a computationally efficient regularization
algorithm under marginality principle (RAMP) that simultaneously selects
the main effects, interaction effects and quadratic effects for high-dimensional data $p\gg n$.
They also verified the interaction selection consistency property of
the two-stage LASSO under some sensible conditions. The stage-wise selection
procedure first performs a main effect selection (by excluding the interaction
terms) to reduce the dimension of variables and then carries out
a joint selection on the reduced dimension of variables, which is
computationally feasible and effective. For example, viewing the sliced
inverse regression \citep{li1991sliced} from a likelihood perspective,
\citet{Jiang:2014hj} suggested a stage-wise variable selection algorithm
(SIRI) via inverse regression, which is able to detect higher order
interactions without imposing any hierarchical structures. \citet{Hao:2014dg}
proposed two stage-wise interaction selection procedures, IFORT and
IFORM, both of which enjoy sure screening property in the first stage.
\citet{fan2016interaction} proposed a method, named the interaction
pursuit, that incorporates both screening and variable selection in
ultra-high dimensions. The method possesses both the sure screening property and the oracle property in the two stages respectively. For some other works on interaction selection,
see \citealt{zhao2009composite,li2012feature,bien2013lasso,hall2014selecting}.
While having the aforementioned good properties, both types of interaction selection approaches have their own
disadvantages as well. The joint selection is usually computational infeasible
(insufficient storage) when $p$ is large; the stage-wise selection,
as pointed out in \citet{Hao:2014dg}, may be very difficult to be
theoretically justified under general conditions.

Although there have been many novel developments on selection of interaction
terms as described above, little work has been done on the estimation
of the regression function when interactions exist. In this paper,
we present some theoretical results on the minimax rate of convergence for estimating
the high-dimensional regression function with interaction terms under three different
hierarchical structures. Regardless of the heredity condition, our results
show that the minimax rate is determined by the maximum of
the total estimation price of the main effects and that of the interaction
effects. Heredity conditions enter the minimax rate of convergence
in terms of the estimation price of the interaction effects, namely
$r_{2}(1+\log(K/r_{2}))/n$, where $r_{2}$ is the number of non-zero
interaction effects and $K$ is the number of eligible candidate interaction
terms under the different heredity conditions. Consequently, a stronger heredity condition
leads to possibly faster minimax rate of convergence. For example,
when the underlying model has no more than $r_{1}$ non-zero main
effects, at most $K=\binom{r_{1}}{2}$ interaction terms are allowed
to enter the model under strong heredity, compared to $K=r_{1}(p_{n}-(r_{1}+1)/2)$
under weak heredity. As will be seen, only in certain situations is the minimax rate improved
by imposing the strong heredity, although
strong heredity allows fewer eligible interaction terms than the other
two heredity conditions. Also, from the perspective of estimation,
there may be no difference in rate of convergence between weak heredity and no heredity in
many situations. Our results provide a complete characterization and comparison of the minimax rates of convergence under the three heredity conditions.

In real applications, since one does not know the true heredity condition behind the data (or practically the best heredity condition to describe the data at the given sample size), it is desirable to construct an estimator that performs optimally no matter which of the three heredity conditions holds. Such an estimator that adapts to the true heredity condition as well as the unknown number of main and interaction effects will be obtained in this paper. 

The remainder of the paper is organized as follows. In Section \ref{sec:Preliminaries},
we introduce the model setup, the loss function and the heredity conditions
for the problem. In Section \ref{sec:Minimax-rate-of}, after stating
the required assumption, we present our main results of the minimax
rate of convergence under strong heredity. The theoretical results under weak heredity and
no heredity are presented in Section \ref{sec:Minimax-rate-of-1}.
Section \ref{subsec:details} provides detailed rates of convergence under different heredity conditions in relation to the sparsity indices, the ambient dimension and the sample size, followed by Section \ref{subsec:implication} where we present some interesting implications of the detailed results. In Section \ref{sec:extension}, we extend our results to quadratic models in  which both quadratic and interaction effects are considered. In Section \ref{sec:Conclusion-and-Discussion}, we construct an adaptive estimator that achieves the minimax rate of convergence without knowledge of the type of the heredity condition or the sparsity indices ($r_1$ and $r_2$). The proofs of our
results and some technical tools are presented in the Appendix.

\section{\label{sec:Preliminaries}Preliminaries}

\paragraph{Model Setup}

\sloppy Suppose the dataset is composed of $(\mathbf{X},\mathbf{Y})$, where
$\mathbf{X}=(\mathbf{x}_{1},...,\mathbf{x}_{p})$ is a $n\times p$
matrix with $n$ observations on $p$ covariates and $\mathbf{Y}=(y_{1},...,y_{n})^{T}$
is the response vector. We start by considering a linear regression
model with both main effects and two-way interaction effects:
\begin{eqnarray}
\mathbf{Y} & = & \mathbf{Z}\beta+\epsilon,\label{eq:0}
\end{eqnarray}
where $\beta=((\beta^{(1)})^{T},(\beta^{(2)})^{T})^{T}$ is the overall
coefficient vector, $\mathbf{Z}=(\mathbf{X},[\mathbf{XX}])\in\mathbb{R}^{n\times(\frac{p^{2}+p}{2})}$
is the full design matrix, and the random noise vector $\epsilon\sim N(\mathbf{0},\sigma^{2}I_{n})$
with known $\sigma$. More specifically, $\beta^{(1)}\in\mathbb{R}^{p}$
and $\beta^{(2)}\in\mathbb{R}^{\binom{p}{2}}$ are the coefficients
of the main effects and the two-way interaction effects respectively.
Here we define $[\mathbf{XX}]=(\mathbf{x}_{1}\circ\mathbf{x}_{2},...,\mathbf{x}_{1}\circ\mathbf{x}_{p},...,\mathbf{x}_{p-1}\circ\mathbf{x}_{p})^{T}$
as the $n\times\binom{p}{2}$ matrix that contains all the two-way
interaction terms, where $\circ$ denotes the point-wise product of
two vectors. 

 In this paper, our focus is on the fixed
design, i.e., the covariates are considered given. Our goal is to estimate the mean regression function by a linear combination of the covariates and interaction terms. 

\paragraph{Loss Function}

Denote $h(\cdot):\mathbb{R}^{(p^{2}+p)/2}\rightarrow\mathbb{R}$ as
the mean regression function, i.e., $h(\mathbf{z})=\mathbf{z}^{T}\beta$
for $\mathbf{z}\in\mathbb{R}^{(p^{2}+p)/2}$. Denote $\hat{h}(\mathbf{z})=\mathbf{z}^{T}\hat{\beta}$ as an estimated function of $h(\mathbf{z})$. In our fixed design setting, we focus on the prediction loss (or the Averaged Squared Error)
$\mathscr{\mathcal{\mathfrak{\mathsf{L}}}}(h,\hat{h}):=\frac{1}{n}||\mathbf{Z}\beta-\mathbf{Z}\hat{\beta}||_{2}^{2}$,
where $\left\Vert \cdot\right\Vert _{2}$ is the Euclidean norm. Set the index sets for the main
effects and the interaction effects as $\mathtt{I_{main}}=\{1,...,p\}$
and $\mathtt{I_{int}}=\{(i,j):1\leq i<j\leq p\}$ respectively. 

Let $\mathtt{I}=(\mathtt{I}_{1},\mathtt{I}_{2})\subset\mathtt{I_{main}}\varotimes\mathtt{I_{int}}$
($\varotimes$ is the Cartesian product) be the index set of a model
with $|\mathtt{I}_{1}|$ non-zero main effects and $|\mathtt{I}_{2}|$
non-zero interaction effects. In this paper, we consider the data
generating model (\ref{eq:0}) with at least two main effects and
one interaction effect purely for convenience, which does not affect the conclusions. Let $\mathbf{Z}_{\mathtt{I}}$
be the $n\times|\mathtt{I}|$ submatrix of $\mathbf{Z}$ that corresponds
to the model index $\mathtt{I}$. Its corresponding least squares
estimator $P_{\mathtt{I}}\mathbf{Y}$ is used to estimate $\mathbf{Z}\beta$,
where $P_{\mathtt{I}}$ is the projection matrix onto the column space
of $\mathbf{Z}_{\mathtt{I}}$. The loss function of using model $\mathtt{I}$
is denoted as $\mathcal{L}(\mathtt{I}):=\frac{1}{n}||P_{\mathtt{I}}\mathbf{Y}-\mathbf{Z}\beta||_{2}^{2}$. 

\paragraph{Heredity Conditions}

Denote the space of all the $p+\binom{p}{2}$-dimensional vectors
with a hierarchical notation of the subscripts as $$\ddot{\mathbb{R}}^{p}=\{\beta\in\mathbb{R}^{p+\binom{p}{2}}|\beta=(\beta_{1},...,\beta_{p},\beta_{1,2},...,\beta_{p-1,p})\}.$$
We refer to $\beta^{(1)}=(\beta_{1},...,\beta_{p})$
as the subvector consisting of the first $p$ elements in $\beta$, and $\beta^{(2)}=(\beta_{1,2},...,\beta_{p-1,p})$
as the subvector containing the rest of the elements. We introduce
the following two vector spaces: 

\[
\ddot{\mathbb{R}}_{weak}^{p}=\left\{ \beta\in\ddot{\mathbb{R}}^{p}|\mathbbm1_{\beta_{i,j}\neq0}\leq \mathbbm1_{\beta_{i}\neq0}\lor \mathbbm1_{\beta_{j}\neq0},1\leq i<j\leq p\right\}
\]
and
\[
\ddot{\mathbb{R}}_{strong}^{p}=\left\{ \beta\in\ddot{\mathbb{R}}^{p}|\mathbbm1_{\beta_{i,j}\neq0}\leq \mathbbm1_{\beta_{i}\neq0}\cdot \mathbbm1_{\beta_{j}\neq0},1\leq i<j\leq p\right\}.
\]
The space $\ddot{\mathbb{R}}_{strong}^{p}$ captures
the strong heredity condition that if the interaction term is in the
model, then both of its corresponding main effects should also be
included. The space $\ddot{\mathbb{R}}_{weak}^{p}$ characterizes the weak
heredity condition that if the interaction is in the model, then at
least one of its main effects should be included. As pointed out in
\citet{hao2016note}, the sign of the main effect coefficients are
not invariant of linear transformation of the covariates individually due to the existence of the
interaction terms. Heredity conditions are consequently meaningless
without the specification of the model parametrization. In our paper,
we stick to the parameterization ${\bf Z}$ and include the no heredity condition by considering the vector space $\ddot{\mathbb{R}}^{p}$. Define the $l_{0}$-norm of a vector $a=(a_{1},...,a_{p})$ as the
number of its non-zero elements, i.e., $\left\Vert a\right\Vert _{0}=\sum_{i=1}^{p}\mathbbm1_{a_{i}\neq0}$.
For a vector space $\mathcal{S}\in\left\{ \ddot{\mathbb{R}}_{strong}^{p},\ddot{\mathbb{R}}_{weak}^{p},\ddot{\mathbb{R}}^{p}\right\} $,
define the corresponding $l_{0}$-$ball$ and $l_{0}$-$hull$ of
$\mathcal{S}$ as 
\begin{equation}
B_{0}(r_{1},r_{2};\mathcal{S})=\left\{ \beta=(\beta^{(1)},\beta^{(2)})\in\mathcal{S},\left\Vert \beta^{(1)}\right\Vert _{0}\le r_{1},\left\Vert \beta^{(2)}\right\Vert _{0}\le r_{2}\right\}\label{eq:bound} 
\end{equation}
 and 
\[
\mathcal{F}_{0}(r_{1},r_{2};\mathcal{S})=\left\{ h:h(\mathbf{z})=\mathbf{z}^{T}\beta,\beta\in B_{0}(r_{1},r_{2};\mathcal{S})\right\} 
\]
respectively. Note that $B_{0}(r_{1},r_{2};\mathcal{S})$ represents
the collection of coefficients $\beta$ with at most $r_{1}$ non-zero
main effects and $r_{2}$ non-zero interaction effects under
a certain hierarchical constraint $\mathcal{S}$. And $\mathcal{F}_{0}(r_{1},r_{2};\mathcal{S})$
denotes the collection of linear combinations of the covariates with coefficients $\beta\in B_{0}(r_{1},r_{2};\mathcal{S})$.
Throughout this paper, we assume that $r_{1}+r_{2}\leq n$ (otherwise
the minimax risk may not converge or the rate may not be optimal), $r_1\geq2$ and $r_2\geq1$.

\paragraph{Minimax Risk}

It is helpful to consider the uniform performance of a modeling procedure when we have plentiful choices of modeling procedures during the analysis
of a statistical problem. The minimax framework seeks an estimator that minimizes the worst performance (in statistical risk) assuming that the truth belongs to a function class
$\mathcal{W}$. The minimax risk we consider is $$\min_{\hat{h}}\max_{h\in\mathcal{W}}E\mathsf{L}(\hat{h},h),$$ where $\hat{h}$ is over all estimators, and $\min$ and $\max$ may refer to $\inf$ and $\sup$, more formally speaking.
In our work, we assume that the true mean regression function has
a hierarchical structure by imposing $\mathcal{W}=\mathcal{F}_{0}(r_{1},r_{2};\mathcal{S})$,
with $\mathcal{S}\in\left\{ \ddot{\mathbb{R}}_{strong}^{p},\ddot{\mathbb{R}}_{weak}^{p},\ddot{\mathbb{R}}^{p}\right\} $. 

In this paper, we will use the notation $b_{n}\succeq a_{n}$
or $a_{n}\preceq b_{n}$ to represent $a_{n}=O(b_{n})$. If both $b_{n}\succeq a_{n}$ and $a_{n}\succeq b_{n}$ hold, we denote $a_{n}\asymp b_{n}$ to indicate that $a_{n}$ and $b_{n}$ are of the same order. If $a_{n}\succeq b_{n}$ holds without $a_{n}\asymp b_{n}$, we use the notation $a_{n}\succ b_{n}$ or $b_{n}\prec a_{n}$.

\section{\label{sec:Minimax-rate-of}Minimax Rate of Convergence under Strong
Heredity}
\subsection{Assumption}\label{subsec:assumption}
We start by stating an assumption required for our result of the minimax rate of convergence under strong heredity. In this paper, we use $p_{n}$ to indicate that the number of main
effects $p$ can go to infinity as $n$ increases. We also allow $r_{1}$ and
$r_{2}$ to increase with the sample size $n$ as well. 
\paragraph{Sparse Reisz Condition (SRC)}
For some $l_{1},l_{2}>0$, there exist constants $b_{1},b_{2}>0$
(not depending on $n$) such that for any $\beta=(\beta^{(1)},\beta^{(2)})$
with $\left\Vert \beta^{(1)}\right\Vert _{0}\le\min(2l_{1},\text{p}_{n})$
and $\left\Vert \beta^{(2)}\right\Vert _{0}\le\min\left(2l_{2},\binom{p_{n}}{2}\right)$,
we have
\begin{equation}
b_{1}\left\Vert \beta\right\Vert _{2}\leq\frac{1}{\sqrt{n}}\left\Vert \mathbf{Z}\beta\right\Vert _{2}\leq b_{2}\left\Vert \beta\right\Vert _{2}. \label{eq:src}
\end{equation}

The SRC assumption requires that the eigenvalues of  $\mathbf{Z}^T_\mathtt{I}\mathbf{Z}_\mathtt{I}$ for any sparse submatrix $\mathbf{Z}_\mathtt{I}$ of $\mathbf{Z}$ are bounded above and away from 0. It was
first proposed by \citet{zhang2008sparsity}. It is similar to the sparse eigenvalue conditions in \citet{zhang2010analysis,raskutti2011minimax},
quasi-isometry condition in \citet{Rigollet:2011fz}; it is also related
to the more stringent restricted isometry property (which requires
the constants $b_{1}$, $b_{2}$ are close to 1) in \citet{candes2007dantzig}.
Such assumptions are standard in the $l_{1}$-regularization analysis
like LASSO and the Dantzig selector. See \citet{bickel2009simultaneous,meinshausen2009lasso,van2007deterministic,koltchinskii2009dantzig}
for more references.

\subsection{Minimax Rate\label{subsec:Upper-bound}}
Now we present our main result of the minimax rate of convergence under strong heredity. A simple estimator is enough for an effective minimax upper bound. Let $\hat{\mathtt{I}}=\arg\min_{\mathtt{I}\in\mathtt{I}^{strong}_{r_1,r_2}}\sum_{i=1}^{n}(Y_{i}-\hat{Y}_{i}^{\mathtt{I}})^{2}$ be the model that minimizes the residual sum of squares over all the models that have exactly $r_1$ non-zero main effects and $r_2$ non-zero interaction effects under strong heredity, denoted as $\mathtt{I}^{strong}_{r_1,r_2}$, where $\hat{\mathbf{Y}}^{\mathtt{I}}=P_{\mathtt{I}}\mathbf{Y}$ is the projection of $\mathbf{Y}$ onto the column space of the design matrix $\mathbf{Z}_{\mathtt{I}}$. For lower bounding the minimax risk, the information-theoretical tool of using Fano's inequality with metric entropy understanding \citep{Yang:1999em} plays an important role in the proof. 

\begin{thm}
Under the Sparse Reisz Condition with $l_{1}=r_{1}\leq p_n\land n$,
$l_{2}=r_{2}\leq\binom{r_{1}}{2}\land n$ and the strong heredity condition
$\mathcal{W}=\mathcal{F}_{0}(r_{1},r_{2};\ddot{\mathbb{R}}_{strong}^{p_{n}})$, the minimax risk is upper bounded by 
\begin{equation}
\min_{\hat{h}}\max_{h\in\mathcal{W}}E\mathsf{L}(\hat{h},h)\leq{\sup_{h\in\mathcal{W}}}E(\mathcal{L}(\hat{\mathtt{I}}))\le\frac{c\sigma^{2}}{n}\left(r_{1}\left(1+\log\frac{p_{n}}{r_{1}}\right)+r_{2}\left(1+\log\frac{\binom{r_{1}}{2}}{r_{2}}\right)\right),
\label{thm:When-,-we1}\end{equation}
where $c$ is a pure constant; the minimax risk is lower bounded by 
\begin{equation}
\min_{\hat{h}}\max_{h\in\mathcal{W}}E\mathsf{L}(\hat{h},h)\geq c_{1}\frac{\sigma^{2}}{n}\left(r_{1}\left(1+\log\frac{p_{n}}{r_{1}}\right)\lor r_{2}\left(1+\log\frac{\binom{r_{1}}{2}}{r_{2}}\right)\right)
\label{thm:If-,-thenlower}
\end{equation}
for some positive constant $c_{1}$ that only depends on the constants
$b_{1}$ and \textbf{$b_{2}$} in the SRC assumption. 
\end{thm}
From the theorem, under the SRC and the strong heredity condition, the minimax rate of convergence scales as: $\min_{\hat{h}}\max_{h\in\mathcal{W}}E\mathsf{L}(\hat{h},h)\asymp \frac{\sigma^{2}}{n}(r_{1}(1+\log\frac{p_{n}}{r_{1}})\lor r_{2}(1+\log(\binom{r_{1}}{2}/r_{2})))$.
\begin{rem}
The term $r_{1}(1+\log(p_{n}/r_{1}))/n=\frac{r_{1}}{n}+\frac{r_1}{n}\log(p_{n}/r_{1})$ reflects two aspects in the estimation of the main effects: the price of searching among $\binom{p_n}{r_1}$ possible models, which is of order $r_{1}\log(p_{n}/r_{1})/n$, and the price of estimating
the $r_1$ main effect coefficients after the search. Thus $r_1(1+\log(p_{n}/r_{1}))/n$ is {\it the total price of estimating the main effects}. Similarly, $r_{2}\left(1+\log\left(\binom{r_{1}}{2}/r_{2}\right)\right)/n$ is {\it the total price of estimating the interaction effects}.
\end{rem}

\begin{rem}
Our result of the upper bound is general and does not require the sparsity condition
of $r_{1}\prec p_{n}$, although it may be needed for fast rate of convergence.
\end{rem}

\section{Minimax Rate of Convergence under Weak Heredity and No Heredity\label{sec:Minimax-rate-of-1}}

Similar results are obtained under weak heredity and no
heredity. The minimax rate of convergence is still determined by the maximum
of the total price of estimating the main effects and that of the interaction effects. When the heredity condition changes, the
total price of estimating the interaction effects may differ, possibly substantially.
\begin{thm}
\label{thm:Under-the-Sparse3}Under the Sparse Reisz Condition with
$l_{1}=r_{1}\leq p_{n}\land n$, $l_{2}=r_{2}\leq(r_{1}p_{n})\land n$ and the weak heredity condition $\mathcal{W}=\mathcal{F}_{0}(r_{1},r_{2};\ddot{\mathbb{R}}_{weak}^{p_{n}})$,
the minimax risk is of order 
\begin{equation}
\min_{\hat{h}}\max_{h\in\mathcal{W}}E\mathsf{L}(\hat{h},h)\asymp\frac{\sigma^{2}}{n}\left(r_{1}\left(1+\log\frac{p_{n}}{r_{1}}\right)\lor r_{2}\left(1+\log\frac{r_{1}\cdot p_{n}}{r_{2}}\right)\right).\label{eq:2}
\end{equation}
\end{thm}
\begin{thm}
\label{thm:Under-the-Sparse4}Under the Sparse Reisz Condition with $l_{1}=r_{1}\leq p_{n}\land n$, $l_{2}=r_{2}\leq\binom{p_{n}}{2}\land n$ and the no heredity condition $\mathcal{W}=\mathcal{F}_{0}(r_{1},r_{2};\ddot{\mathbb{R}}^{p_{n}})$,
the minimax risk is of order 
\begin{equation}
\min_{\hat{h}}\max_{h\in\mathcal{W}}E\mathsf{L}(\hat{h},h)\asymp\frac{\sigma^{2}}{n}\left(r_{1}\left(1+\log\frac{p_{n}}{r_{1}}\right)\lor r_{2}\left(1+\log\frac{\binom{p_{n}}{2}}{r_{2}}\right)\right).\label{eq:3}
\end{equation}
\end{thm}

\section{Comparisons and Insights\label{sec:Comparisons-and-Insights5}}

In this section, we summarize the consequences of our main results in three
scenarios for an integrated understanding. For brevity, we introduce the following notation.
For $a,b\in\mathbb{N}^{+}$ and $a\geq b$, define the quantity $\xi_{b}^{a}:=b(1+\log(a/b)).$
The total price of estimating the main effects and the interaction effects are then denoted as $\sigma^{2}\xi_{r_{1}}^{p_{n}}/n$ and $\sigma^{2}\xi_{r_{2}}^{K}/n$ respectively,
where $K$ depends on $p_{n}$, $r_{1}$ and the heredity condition. We also use the notation $K_{\mathcal{S}}$ (\ref{eq:ks}) to indicate that $K$ depends on the heredity condition $\mathcal{S}$. 
Let 
\[
\mathcal{M}(\mathcal{S}):=\min_{\hat{h}}\underset{h\in\mathcal{F}_{0}(r_{1},r_{2};\mathcal{S})}{\max E\mathsf{L}(\hat{h},h)}
\]
denote the minimax risk under the heredity condition $\mathcal{S}$.

\subsection{Detailed Rates of Convergence}
Since the minimax rate of convergence depends on the maximum of
$\xi_{r_{1}}^{p_{n}}$ and $\xi_{r_{2}}^{K}$, we discuss the cases 
where one of the two quantities is greater than the other. 
\subsubsection*{Scenario 1: $\boldsymbol{r_{2}\preceq r_{1}}$}\label{subsec:details}

When there are more main effects than interaction effects in the sense that $r_{2}\preceq r_{1}$, the minimax
rate of convergence is not affected by the heredity conditions. When $\log (p_{n}/r_{1})\succeq\log r_{1}$, we always have $\xi_{r_{1}}^{p_{n}}\succeq\xi_{r_{2}}^{K}$ regardless of the heredity conditions. When $\log (p_{n}/r_{1})\prec\log r_{1}$, it depends on the order of $r_2$ to further decide which estimation price is larger. When $\log (p_{n}/{r_{1}})\prec\log r_{1}$, let $r_{*}$ be such that $\xi_{r_{1}}^{p_{n}}\asymp\xi_{r_{*}}^{r_{1}^{2}}$. If $r_{*}\succeq r_{2}$, we have $\xi_{r_{1}}^{p_{n}}\succeq\xi_{r_{2}}^{K}$; otherwise $\xi_{r_{1}}^{p_{n}}\prec\xi_{r_{2}}^{K}$. 

In summary, given that $r_{2}\preceq r_{1}$, the minimax risk is of order 
\[
\mathcal{M}(\mathcal{S})\asymp\begin{cases}
\frac{\sigma^{2}}{n}\xi_{r_{2}}^{r_{1}^{2}}, & \textrm{if }r_{*}\preceq r_{2}\preceq r_1\textrm{ and }\log\frac{p_{n}}{r_{1}}\prec\log r_{1},\\
\frac{\sigma^{2}}{n}\xi_{r_{1}}^{p_{n}}, & \textrm{otherwise,}
\end{cases}
\]
for $\mathcal{S}\in\left\{ \ddot{\mathbb{R}}_{strong}^{p},\ddot{\mathbb{R}}_{weak}^{p},\ddot{\mathbb{R}}^{p}\right\}$. 
\begin{rem}
This scenario also includes the special case
when \textbf{$p_{n}=O(1)$}, where we must have $r_{1}=O(1)$ and
$r_{2}=O(1)$. The minimax rate of convergence is of the standard parametric
order $1/n$ regardless of the heredity conditions.
\end{rem}

\paragraph{Scenario 2:\textmd{ $\boldsymbol{r_{1}\preceq r_{2}}$} and $\boldsymbol{\log p_{n}\preceq r_{1}}$}
When there exist more
interaction terms, i.e., $r_{1}\preceq r_{2}$, under weak or no heredity, the quantity $\xi_{r_{2}}^{K}$
is always no less than (in order) $\xi_{r_{1}}^{p_{n}}$. 

For strong heredity, we discuss case by case. When $\log (p_n/r_1)\prec \log r_1$, we always have $\xi_{r_{1}}^{p_{n}}\preceq\xi_{r_{2}}^{r_1^2}$. When $\log (p_n/r_1)\succeq \log r_1$, it depends on the order of $r_2$ to decide which estimation price is larger in terms of order.  When $\log (p_{n}/{r_{1}})\succeq\log r_{1}$, let $r^{\prime}_{*}$ be such that $\xi_{r_{1}}^{p_{n}}\asymp\xi_{r^{\prime}_{*}}^{r_{1}^{2}}$. If $r_{2}\succeq r^{\prime}_{*}$, we have $\xi_{r_{1}}^{p_{n}}\preceq\xi_{r_{2}}^{r_1^2}$; otherwise $\xi_{r_{1}}^{p_{n}}\succ\xi_{r_{2}}^{r_1^2}$. In summary, given that  $r_{1}\preceq r_{2}$ and $\log p_n\preceq r_1$, the minimax risk is of order
\[
\mathcal{M}(\ddot{\mathbb{R}}_{strong}^{p_{n}})\asymp\begin{cases}
\frac{\sigma^{2}}{n}\xi_{r_{1}}^{p_{n}}, & \textrm{if } r_{1}\preceq r_{2}\preceq r^{\prime}_{*}\textrm{ and }\log\frac{p_{n}}{r_{1}}\succeq\log r_{1},\\
\frac{\sigma^{2}}{n}\xi_{(r_{2}\land r_{1}^{2})}^{r_{1}^{2}}, & \textrm{otherwise,}
\end{cases}
\]

\[
\mathcal{M}(\ddot{\mathbb{R}}_{weak}^{p_{n}})\asymp\frac{\sigma^{2}}{n}\xi_{(r_{2}\land r_{1}p_{n})}^{r_{1}p_{n}},
\]

\[
\mathcal{M}(\ddot{\mathbb{R}}^{p_{n}})\asymp\frac{\sigma^{2}}{n}\xi_{(r_{2})}^{p_{n}^{2}}.
\]

\begin{rem}
The term $\xi_{(r_{2}\land K)}^{K}$ deals with the case where $r_{2}$ is inactive in the sense that $r_{2}$ exceeds $K$ under the specific heredity condition.
For example, with $r_{2}\geq \binom{r_{1}}{{2}}$, the upper bound $r_2$ in (\ref{eq:bound}) does not provide any new information of the number of non-zero interaction effects for strong heredity. Thus the $l_{0}$-ball $B_{0}(r_{1},r_{2};\ddot{\mathbb{R}}_{strong}^{p})$
is automatically reduced to a subset $B_{0}(r_{1},\binom{r_{1}}{2};\ddot{\mathbb{R}}_{strong}^{p})$. 
\end{rem}

\paragraph{Scenario 3: $\boldsymbol{r_{1}\preceq r_{2}}$ and $\boldsymbol{\log p_{n}\succeq r_{1}}$}

When the number of the main effects $p_{n}$ is at least exponentially
as many as the non-zero main effects in the sense that $\log p_{n}\succeq r_{1}$, $\xi_{r_{1}}^{p_{n}}$
is always no less than $\xi_{r_{2}}^{K}$ in terms of order. In fact, in this scenario, the results
of the minimax rates under weak or no heredity are exactly the same as those
in Scenario 2. For completeness, we still present the results. Specifically,
the minimax risk is of order

\[
\mathcal{M}(\ddot{\mathbb{R}}_{strong}^{p_{n}})\asymp \frac{\sigma^{2}}{n}\xi_{r_{1}}^{p_{n}},
\]

\[
\mathcal{M}(\ddot{\mathbb{R}}_{weak}^{p_{n}})\asymp\frac{\sigma^{2}}{n}\xi_{(r_{2}\land r_{1}p_{n})}^{r_{1}p_{n}},
\]

\[
\mathcal{M}(\ddot{\mathbb{R}}^{p_{n}})\asymp\frac{\sigma^{2}}{n}\xi_{r_{2}}^{p_{n}^{2}}.
\]

\subsection{Interesting Implications}\label{subsec:implication}
\begin{enumerate}
\item Comparing the results for weak heredity and no heredity, we may or may not have distinct rates of convergence. When there exists a small constant $c>0$ such that $\log r_{2}\leq(1-c)\cdot\log(r_{1}p_{n})$ for large enough $n$, there is no difference
between weak heredity and no heredity from the perspective of rate of convergence in estimation. It still remains an open question how they are different for the problem of model identification. Without the above relationship between $r_1$ and $r_2$, there is no guarantee that the rates of convergence are the same under weak heredity and no heredity. For example, when $r_{2}= r_{1}p_{n}/\log r_{1}$, if in addition we have $r_1=p_n\leq n^{1/2}$, 
the minimax rates are the same under weak and no heredity,
at $\mathcal{M}(\ddot{\mathbb{R}}_{weak}^{p_{n}})\asymp \mathcal{M}(\ddot{\mathbb{R}}^{p_{n}})\asymp r_{1}p_{n}\log\log r_{1}/(n\log r_{1})$. In contrast, if instead we have $r_{1}= \sqrt{p_{n}}$, then the minimax rates are different, with $\mathcal{M}(\ddot{\mathbb{R}}_{weak}^{p_{n}})\asymp r_{1}p_{n}\log\log r_{1}/(n\log r_{1})$ and $\mathcal{M}(\ddot{\mathbb{R}}^{p_{n}})\asymp r_{1}p_{n}/n$.

\item Heredity conditions do not affect the rates of convergence in some situations.
For example, when there exist more main effects than interaction effects (Scenario 1), the minimax rates of convergence are the same
under all three heredity conditions. 
\item From the detailed rates of convergence, under any of the three heredity conditions, the estimation of the interaction terms $\xi^{K}_{r_2}/n$ may become the dominating part. There are two different reasons why the price of estimating the interaction terms becomes higher than that for the main effect terms. One is that the number of interaction terms is more than that of the main effect terms. The other reason is that although the main effect terms outnumber the interaction terms, the ambient dimension is so large that even estimating a small number of the interaction terms is more challenging than estimating the main effects.
\item How much can the rate of convergence be improved by imposing strong
heredity? We quantify this improvement by taking the ratio of two
minimax rates of convergence given the ambient dimension $p_n$, i.e., $\mathcal{M}(\ddot{\mathbb{R}}_{strong}^{p_{n}})/\mathcal{M}(\ddot{\mathbb{R}}_{weak}^{p_{n}})$
and $\mathcal{M}(\ddot{\mathbb{R}}_{strong}^{p_{n}})/\mathcal{M}(\ddot{\mathbb{R}}^{p_{n}})$. In Scenario 2 ($r_{1}\preceq r_{2}$ and
$\log p_{n}\preceq r_{1}$), we have $\mathcal{M}(\ddot{\mathbb{R}}_{strong}^{p_{n}})/\mathcal{M}(\ddot{\mathbb{R}}_{weak}^{p_{n}})\succeq \log p_{n}/p_{n}$,
where the maximal improvement happens when $r_{1}\asymp \log p_{n}$
and $r_{2}\asymp r_{1}p_{n}$. That is, the minimax rate of convergence
under strong heredity is up to $\log p_{n}/p_{n}$ times faster than that
under weak heredity. Similarly we have $\mathcal{M}(\ddot{\mathbb{R}}_{strong}^{p_{n}})/\mathcal{M}(\ddot{\mathbb{R}}^{p_{n}})\succeq \log^{2}p_{n}/p_{n}^{2}$,
where the maximal improvement $\log^{2}p_{n}/p_{n}^{2}$ happens at
$r_{1}\asymp \log p_{n}$ and $r_{2}\asymp p_{n}^{2}$. 
\item In Scenario 3 ($r_{1}\preceq r_{2}$ and $\log p_{n}\succeq r_{1}$),
the improvement $\mathcal{M}(\ddot{\mathbb{R}}_{strong}^{p_{n}})/\mathcal{M}(\ddot{\mathbb{R}}_{weak}^{p_{n}})\succeq\log p_{n}/p_{n}$,
where the maximal improvement happens when $r_{2}\succeq r_{1}p_{n}$.
In this scenario, the maximal improvement of the minimax rate from weak
heredity to strong heredity depends on the ambient dimension
$p_{n}$. In other words, the larger the ambient dimension is, the
more improvement of minimax rate of convergence we have from weak
heredity to strong heredity. Similarly we have $\mathcal{M}(\ddot{\mathbb{R}}_{strong}^{p_{n}})/\mathcal{M}(\ddot{\mathbb{R}}^{p_{n}})\succeq\log p_{n}/p_{n}^{2}$,
where the equality holds if $r_{1}=O(1)$ and $r_{2}\asymp p_{n}^{2}$.
\item If $r_{2}$ is active for all three heredity conditions, i.e., $r_2\leq \binom{r_1}{2}$, the maximal improvement of minimax rate from weak/no heredity to strong heredity turns out to be consistent. That is, $\mathcal{M}(\ddot{\mathbb{R}}_{strong}^{p_{n}})/\mathcal{M}(\ddot{\mathbb{R}}_{weak}^{p_{n}})\asymp\mathcal{M}(\ddot{\mathbb{R}}_{strong}^{p_{n}})/\mathcal{M}(\ddot{\mathbb{R}}^{p_{n}})\succeq1/\log p_{n}$,
where the maximal improvement happens at $r_{1}\asymp \log p_{n}$
and $r_{2}\asymp r_{1}^{2}$.
\end{enumerate}

\section{Extension to Quadratic Models\label{sec:extension}}

Our aforementioned results do not consider quadratic effects. When
both quadratic and two-way interaction effects are included in a model (called a quadratic model), it is
easy to see the rates of convergence in the theorems still apply under
both strong heredity and weak heredity. However, in the case of no
heredity, the number of quadratic terms enters into the minimax rate. Assume one model has at most $r_{3}$ extra non-zero quadratic terms. We need
the following assumption. 

\paragraph{Sparse Reisz Condition 2 (SRC2)}

For some $l_{1},l_{2},l_{3}>0$, there exist constants $b_{1},b_{2}>0$
(not depending on $n$) such that for any $\beta=(\beta^{(1)},\beta^{(2)},\beta^{(3)})$
with $\left\Vert \beta^{(1)}\right\Vert _{0}\le\min(2l_{1},\text{p}_{n})$,
$\left\Vert \beta^{(2)}\right\Vert _{0}\le\min(2l_{2},\binom{p_{n}}{2})$
and $\left\Vert \beta^{(3)}\right\Vert _{0}\le\min(2l_{3},p_{n})$,
we have
\[
b_{1}\left\Vert \beta\right\Vert _{2}\leq\frac{1}{\sqrt{n}}\left\Vert \mathbf{Z}^*\beta\right\Vert _{2}\leq b_{2}\left\Vert \beta\right\Vert _{2},
\]where $\mathbf{Z}^*=(\mathbf{X},[\mathbf{XX}],\mathbf{X^2})$ is the new design matrix, with $\mathbf{X^2}$ representing the $n\times p$ matrix that contains all the quadratic terms.

Next we state the minimax results for quadratic models. Strong heredity
and weak heredity are exactly the same condition since a quadratic term has only one corresponding main effect term. That is, both strong and
weak heredity require that if a quadratic term $X_{1}^{2}$ has a
non-zero coefficient, then $X_{1}$ must also have a non-zero coefficient.
Similarly, under SRC2 with $l_{1}=r_{1},l_{2}=r_{2},l_{3}=r_{3}$, the minimax rate of convergence under strong/weak
heredity for the quadratic model stays the order 
\begin{equation}
\frac{\sigma^{2}}{n}\left(r_{1}(1+\log\frac{p_{n}}{r_{1}})\lor r_{2}(1+\log\frac{\binom{r_{1}}{2}}{r_{2}})\right);\label{eq:interactionstrong}
\end{equation}
under no heredity, its order becomes

\begin{equation}
\frac{\sigma^{2}}{n}\left(\bar{r}(1+\log\frac{p_{n}}{\bar{r}})\lor r_{2}(1+\log\frac{\binom{p_{n}}{2}}{r_{2}})\right),\label{eq:interactionno}
\end{equation}
where $\bar{r}=r_{1}\lor r_{3}$.

\section{Adaptation to Heredity Conditions and Sparsity Indices\label{sec:Conclusion-and-Discussion}}

In the previous sections, we have determined the minimax rates of convergence for estimating the linear regression function with interactions under different sizes of sparsity indices $r_1,r_2$ and heredity conditions $\mathcal{S}$. These results assume that $r_1$, $r_2$ and $\mathcal{S}$ are known. However, in practice, we usually have no prior information about the underlying heredity condition nor the sparsity constraints. Thus it is necessary and appealing to build an estimator that adaptively achieves the minimax rate of convergence without the knowledge of $\mathcal{S}$, $r_1$ and $r_2$. We construct such an adaptive estimator as below.

To achieve our goal, we consider one specific model and three types of models together as the candidate models: $$\bar{\mathcal{F}}= \{\mathtt{I}_{p_{n},(p_{n}^{2}-p_{n})/2}\}\cup\{\mathtt{I}_{k_{1},k_{2}}^{strong}\}\cup\{\mathtt{I}_{k_{1},k_{2}}^{weak}\}\cup\{\mathtt{I}_{k_{1},k_{2}}^{no}\},$$ where $\mathtt{I}_{p_{n},(p_{n}^{2}-p_{n})/2}$ denotes the full model with $p_n$ main effects and all the $\binom{p_n}{2}$ interaction effects. It is included so that the risk of our estimator will not be worse than order $R_{\mathbf{Z}}/n$, in which $R_{\mathbf{Z}}$ is the rank of the full design matrix. With a slight abuse of the notation, we use
$\mathtt{I}_{k_{1},k_{2}}^{strong}$, $\mathtt{I}_{k_{1},k_{2}}^{weak}$
and $\mathtt{I}_{k_{1},k_{2}}^{no}$ to represent a model with $k_{1}$
main effects and $k_{2}$ interaction effects under strong
heredity, weak heredity and no heredity respectively. Note that some models appear more than once in $\bar{\mathcal{F}}$, which does not cause any problem for the goal of estimating
the regression function. The details of the range of $k_1$ and $k_2$ for each model class are shown in (\ref{eq:cstrong}), (\ref{eq:cweak}) and (\ref{eq:cno}).

To choose a model from the candidate set, we apply the ABC criterion in \citet{yang1999model}. For a model $\mathtt{I}$ in $\bar{\mathcal{F}}$, the criterion value is
\begin{equation}
ABC(\mathtt{I})=\sum_{i=1}^{n}(Y_{i}-\hat{Y}_{i}^{\mathtt{I}})^{2}+2r_{\mathtt{I}}\sigma^{2}+\lambda\sigma^{2}C_{\mathtt{I}},\label{eq:1}
\end{equation}
where $\hat{\mathbf{Y}}^{\mathtt{I}}=P_{\mathtt{I}}\mathbf{Y}$ is
the projection of $\mathbf{Y}$ onto the column space of the design
matrix $\mathbf{Z}_{\mathtt{I}}$ with rank $r_{\mathtt{I}}$, $C_{\mathtt{I}}$ is the descriptive complexity
of model $\mathtt{I}$ and $\lambda>0$ is a constant. The model descriptive
complexity satisfies $C_{\mathtt{I}}>0$ and $\sum_{\mathtt{I}\in\bar{\mathcal{F}}}\exp(-C_{\mathtt{I}})\leq1$. 

The model descriptive complexity is crucial in building the adaptive model. Let $\pi_0, \pi_1,\pi_2,\pi_3\in(0,1)$ be four constants such that $\pi_0+\pi_1+\pi_2+\pi_3=1$. Set $C_{\mathtt{I}_{p_{n},(p_{n}^{2}-p_{n})/2}}=-\log \pi_0$ for the full model,
\begin{equation}
C_{\mathtt{I}_{k_{1},k_{2}}^{strong}}=-\log \pi_1+\log(p_{n}\land n)+\log\left(\binom{k_{1}}{2}\land n\right)+\log\binom{p_{n}}{k_{1}}+\log\binom{\binom{k_{1}}{2}}{k_{2}}\label{eq:cstrong}
\end{equation}
for $1\leq k_{1}\le p_{n}\land n\textrm{ and }0\leq k_{2}\le \binom{k_{1}}{2}\land n $,

\begin{equation}
 C_{\mathtt{I}_{k_{1},k_{2}}^{weak}}=-\log \pi_2+\log(p_{n}\land n)+\log\left(K\land n\right)+\log\binom{p_{n}}{k_{1}}+\log\binom{K}{k_{2}}\label{eq:cweak}
\end{equation}
with $K=k_{1}p_{n}-\binom{k_{1}}{2}-k_{1}$ for $1\leq k_{1}\leq p_{n}\land n\textrm{ and }0\leq k_{2}\leq K\land n$, and
\begin{equation}
C_{\mathtt{I}_{k_{1},k_{2}}^{no}}=-\log \pi_3+\log(p_{n}\land n)+\log\left(\binom{p_{n}}{2}\land n\right)+\log\binom{p_{n}}{k_{1}}+\log\binom{\binom{p_{n}}{2}}{k_{2}}\label{eq:cno},
\end{equation}
for $1\leq k_{1}\leq p_{n}\land n\textrm{ and }0\leq k_{2}\leq\binom{p_{n}}{2}\land n$. This complexity assignment recognizes that there are three types of models under the different heredity conditions.

Let $\hat{\mathtt{I}}=\arg\min_{\mathtt{I}\in\bar{\mathcal{F}}}ABC(\mathtt{I})$
denote the model that minimizes the ABC criterion over
the candidate model set $\bar{\mathcal{F}}$ and $\hat{\mathbf{Y}}^{\hat{\mathtt{I}}}:=P_{\hat{\mathtt{I}}}\mathbf{Y}$ denote
the least squares estimate of $\mathbf{Y}$ using the model $\hat{\mathtt{I}}$. Then we have the following oracle inequality. 
\begin{thm}\label{thm:7}
When $\lambda\ge5.1/\log2$, the worst risk
of the ABC estimator $\hat{\mathbf{Y}}^{\hat{\mathtt{I}}}$ is upper bounded by
\[
\underset{h\in\mathcal{F}_{0}(r_{1},r_{2};\mathcal{S})}{\sup E(\mathcal{L}(\hat{\mathtt{I}}))}\le\frac{c\sigma^{2}}{n}\left[R_{\mathbf{Z}}\land\left(r_{1}\left(1+\log\frac{p_{n}}{r_{1}}\right)+r_{2}\left(1+\log\frac{K_{\mathcal{S}}}{r_{2}}\right)\right)\right],
\]
with 
\begin{equation}
K_{\mathcal{S}}=\begin{cases}
\binom{r_{1}}{2}, & \textrm{if }\mathcal{S}=\ddot{\mathbb{R}}_{strong}^{p},\\
r_{1}p_{n}, & \textrm{if }\mathcal{S}=\ddot{\mathbb{R}}_{weak}^{p},\\
\binom{p_{n}}{2}, & \textrm{if }\mathcal{S}=\ddot{\mathbb{R}}^{p},
\end{cases}\label{eq:ks}
\end{equation}
where $R_{\mathbf{Z}}$ is the rank of the full design matrix $\mathbf{Z}$ and the constant $c$ only depends on the constant $\lambda$. 
\end{thm}
From the theorem, without any prior knowledge of the sparsity indices, the constructed ABC estimator adaptively achieves the minimax upper bound regardless of the heredity conditions. The result also indicates a major difference between estimation and model identification. For estimation, from the result, we are able to achieve adaptation with respect to the heredity condition without any additional assumption. For model identification, although we are not aware of any work that addresses the task of adaptation over the unknown heredity nature, it seems certain that much stronger assumptions than those for consistency under an individual heredity condition will be necessary to achieve adaptive selection consistency. Achieving adaptive model selection consistency under different types of conditions remains an important open problem on model selection theory and methodology.
\begin{rem}
We do not require any assumptions on the relationship among the variables for the upper bound in the theorem. In particular, the variables may be arbitrary correlated.
\end{rem}

\begin{rem}
The order $R_{\mathbf{Z}}/n$ is achievable when
we use the projection estimator from the full model. Thus the minimax rate of convergence
is no slower than the order $R_{\mathbf{Z}}/n$. As is known, the
rank of the design matrix plays an important role in determining the
minimax rate of convergence under fixed design \citep{yang1999model,Rigollet:2011fz,wang2014adaptive}.
For our result, when $p_{n}$, $r_{1}$ and $r_{2}$ together make the total estimation price of the true model
small enough, the upper bound will be improved from $R_{\mathbf{Z}}/n$
to $(r_{1}(1+\log(p_{n}/r_{1}))\lor r_{2}(1+\log(\binom{r_{1}}{2}/r_{2})))/n$.
\end{rem}

\begin{rem}
The ABC estimator may not be practical when $p_n$ is large. In such case, stochastic search instead of all subset selection can be used for implementation.
\end{rem}
\begin{rem}
The term $``R_{\mathbf{Z}}\land"$ automatically applies to the lower bound under whichever heredity condition, since under the SRC assumption, it intrinsically requires that $r_{1}(1+\log(p_{n}/r_{1}))\lor r_{2}(1+\log(\binom{r_{1}}{2}/r_{2}))$ is no larger than $R_{\mathbf{Z}}$ in terms of order. Otherwise, the lower bound $(r_{1}(1+\log(p_{n}/r_{1}))\lor r_{2}(1+\log(\binom{r_{1}}{2}/r_{2})))/n$
by our proof will exceed the upper bound $R_{\mathbf{Z}}/n$, which
leads to a contradiction. We give a specific example in Appendix \ref{subsec:An-example-when}
to illustrate this requirement.
\end{rem}

\appendix

\section{Proof of the upper bound in Theorem 3.1}
\subsection*{Proof of (\ref{thm:When-,-we1})}
Recall that $h(\mathbf{z})=\mathbf{z}^{T}\beta$ and $\hat{h}(\mathbf{z})=\mathbf{z}^{T}\hat{\beta}$. Set $\mathbf{h}_{\mathtt{I}}:=P_{\mathtt{I}}\mathbf{h}$ as the estimator by model $\mathtt{I}$, where we use the bold-face $\mathbf{h}=(h(\mathbf{z}_{1}^{T}),...,h(\mathbf{z}_{n}^{T}))^{T}$
to denote the mean regression function vector and $\mathbf{z}_{i}$
is the $i$-th row of the full design matrix $\mathbf{Z}$. 
We first prove that $\hat{\mathtt{I}}$ is equivalently an ABC estimator over the candidate set we consider.
The SRC assumption with $l_1=r_1$, $l_2=r_2$ assures that $r_1+r_2\leq n$. It follows that, for any model $\mathtt{I}=(\mathtt{I}_1,\mathtt{I}_2)$ with $|\mathtt{I}_1|_0=r_1$, $|\mathtt{I}_2|_0=r_2$, the corresponding submatrix $\mathbf{Z}_{\mathtt{I}}$ is full rank, i.e., $r_{\mathtt{I}}=r_1+r_2$. Thus,  
\begin{eqnarray*}
\hat{\mathtt{I}} & = & \arg\min_{\mathtt{I}\in\mathcal{F}}\sum_{i=1}^{n}(Y_{i}-\hat{Y}_{i}^{\mathtt{I}})^{2}\\
 & = & \arg\min_{\mathtt{I}\in\mathcal{F}}\sum_{i=1}^{n}(Y_{i}-\hat{Y}_{i}^{\mathtt{I}})^{2}+2r_{\mathtt{I}}\sigma^{2}+\lambda\sigma^{2}C_{\mathtt{I}}\\
 & = & \arg\min_{\mathtt{I}\in\mathcal{F}}ABC(\mathtt{I}),
\end{eqnarray*}
where $\mathcal{F}$ is the collection of models that have $r_1$ non-zero main effects and $r_2$ non-zero interaction effects with $0\leq r_1\leq p_n, 1\leq r_2\leq \binom{r_1}{2}$, and all the models in $\mathcal{F}$ share the same model descriptive complexity $$C_{\mathtt{I}_{r_{1},r_{2}}^{strong}}=\log\binom{p_{n}}{r_{1}}+\log\binom{\binom{r_{1}}{2}}{r_{2}}.
$$ The ABC criterion and the model descriptive complexity are introduced near (\ref{eq:1}). Therefore, $\hat{\mathtt{I}}$ is an ABC estimator over the candidate set $\mathcal{F}$.

Next we prove the upper bound. Since $\hat{\mathtt{I}}$ is an ABC estimator over the candidate set $\mathcal{F}$, by Theorem 1 in \citet{yang1999model}, we have:
\begin{equation}
E(\mathcal{L}(\hat{\mathtt{I}}))\leq c\underset{\mathtt{I}\in\mathcal{F}}{\inf}\left(\frac{1}{n}\left\Vert \mathbf{h}_{\mathtt{I}}-\mathbf{h}\right\Vert _{2}^{2}+\frac{\sigma^{2}r_{\mathtt{I}}}{n}+\frac{\lambda\sigma^{2}C_{\mathtt{I}}}{n}\right),\label{eq:yang}
\end{equation}
where $c$ is a positive constant that depends on the constant $\lambda$
only.  When $h\in\mathcal{W}=\mathcal{F}_{0}(r_{1},r_{2};\ddot{\mathbb{R}}_{strong}^{p_{n}})$,
there exists a specific model in $\mathcal{F}$ such that the projection estimator of this model is equal to $\mathbf{h}$. We consider the RHS of (\ref{eq:yang}) evaluated
at such a model, where we still denote it as $\mathtt{I}_{r_{1},r_{2}}$ for convenience. Thus,
\begin{eqnarray*}
E(\mathcal{L}(\hat{\mathtt{I}})) & \leq & c\left(\left\Vert \mathbf{h}_{\mathtt{I}_{r_{1},r_{2}}}-\mathbf{h}\right\Vert ^{2}+\frac{\sigma^{2}r_{\mathtt{I}_{r_{1},r_{2}}}}{n}+\frac{\lambda\sigma^{2}C_{\mathtt{I}_{r_{1},r_{2}}}}{n}\right)\\
 & = & \underset{(i)}{\underbrace{\frac{c}{n}\left(\sigma^{2}r_{\mathtt{I}_{r_{1},r_{2}}}+\lambda\sigma^{2}C_{\mathtt{I}_{r_{1},r_{2}}}\right)}}.
\end{eqnarray*}
The term $(i)$ is bounded as follows:
\begin{eqnarray*}
(i) & \leq & \frac{c_{1}\lambda}{n}\sigma^{2}\left(\frac{1}{\lambda}\left(r_{1}+r_{2}\right)+\log\binom{p_{n}}{r_{1}}+\log\binom{\binom{r_{1}}{2}}{r_{2}}\right)\\
 & \leq & \frac{c_{1}\lambda}{n}\sigma^{2}\left(\frac{1}{\lambda}\left(r_{1}+r_{2}\right)+r_{1}\left(1+\log\frac{p_{n}}{r_{1}}\right)+r_{2}\left(1+\log\frac{\binom{r_{1}}{2}}{r_{2}}\right)\right)\\
 & \leq & \frac{c_{2}}{n}\sigma^{2}\left(r_{1}\left(1+\log\frac{p_{n}}{r_{1}}\right)+r_{2}\left(1+\log\frac{\binom{r_{1}}{2}}{r_{2}}\right)\right).
\end{eqnarray*}
Therefore,
\[
 E(\mathcal{L}(\hat{\mathtt{I}}))\leq\frac{c_{2}\cdot\sigma^{2}}{n}\left(r_{1}\left(1+\log\frac{p_{n}}{r_{1}}\right)+r_{2}\left(1+\log\frac{\binom{r_{1}}{2}}{r_{2}}\right)\right).
\]
Thus we have
\[
 \min_{\hat{h}}\max_{h\in\mathcal{W}}E\mathsf{L}(\hat{h},h)\leq \max_{h\in\mathcal{W}}E(\mathcal{L}(\hat{\mathtt{I}}))\leq\frac{c_{2}\cdot\sigma^{2}}{n}\left(r_{1}\left(1+\log\frac{p_{n}}{r_{1}}\right)+r_{2}\left(1+\log\frac{\binom{r_{1}}{2}}{r_{2}}\right)\right),
\]
where the above $c_{1}$, $c_{2}$ are universal constants.

\section{Proof of the lower bound in Theorem 3.1}

Before stating the proof of (\ref{thm:If-,-thenlower}), we
introduce the local metric entropy, two important sets that aid the understanding of the metric entropy of the regression function space, together with the lemmas in relation to these two sets. 

\subsection{Metric Entropy}

Metric entropy plays a central role in minimax theory, through the concepts of packing and covering. It provides
a way to understand the ``cardinality'' of a set with infinitely
many elements. In deriving the lower bound, information theoretic techniques play
a key role, such as the local metric entropy, Fano's inequality, Shannon's mutual information and
Kullback\textendash Leibler divergence. We begin by introducing the definition of
the local metric entropy.

\begin{defn}[Local Metric Entropy]
 Given a metric space $(\mathcal{X},\rho)$, let $B(x,\epsilon)=\{x^{\prime}\in\mathcal{X}|\rho(x,x^{\prime})\leq\epsilon\}$
be a $\epsilon$-ball around $x$. For $0<a<1$, the \textbf{$a$-local
$\epsilon$-entropy at $x$}, denoted as $\log M_{x}^{a}\left(\epsilon;\mathcal{X},\rho\right)$,
is defined as the $a\epsilon$-packing entropy of $B(x,\epsilon)$.
The \textbf{$a$-local $\epsilon$-entropy}, denoted as $\log M_{\mathtt{local}}^{a}\left(\epsilon;\mathcal{X},\rho\right)$,
is then defined as the maximum (or supremum if maximum does not exist) of $\log M_{x}^{a}(\epsilon;\mathcal{X},\rho)$
over all $x$ in $\mathcal{X}$, i.e., $\log M_{\mathtt{local}}^{a}\left(\epsilon;\mathcal{X},\rho\right)=\underset{x\in\mathcal{X}}{\max}\log M_{x}^{a}\left(\epsilon;\mathcal{X},\rho\right)$.
\end{defn}

\subsection{Important Subsets}

Set the Hamming distance between any two vectors $v,v^{\prime}\in\mathbb{R}^{d}$
as $\rho_{H}(v,v^{\prime})=\sum_{i=1}^{d}\mathbbm1_{v_{i}\neq v_{i}^{\prime}}$.
Consider the set
\[
\mathcal{H}=\left\{ \beta\in\ddot{\mathbb{R}}_{strong}^{p_{n}}:\beta\in\{-1,0,1\}^{p_{n}+\binom{p_{n}}{2}},\left\Vert \beta^{(1)}\right\Vert _{0}\leq r_{1},\left\Vert \beta^{(2)}\right\Vert _{0}\leq r_{2}\right\}
\]
and let $\mathcal{H}_{1}$ denote a subset of $\mathcal{H}$ where the
the first $r_{1}$ coordinates are fixed, i.e.,
\[
\mathcal{H}_{1}=\left\{ \beta\in\mathcal{H}:\beta^{(1)}=(\underset{r_{1}}{\underbrace{1,...,1}},\underset{p_{n}-r_{1}}{\underbrace{0,...,0}}),\left\Vert \beta^{(2)}\right\Vert _{0}=r_{2}\right\} .
\]
Let $\mathcal{H}_{2}$ denote another subset of $\mathcal{H}$ where
no interaction effect exists, i.e.,

\[
\mathcal{H}_{2}=\left\{ \beta\in\mathcal{H}:\left\Vert \beta^{(1)}\right\Vert _{0}=r_{1},\left\Vert \beta^{(2)}\right\Vert _{0}=0\right\} ,
\]

The following two lemmas of the metric entropy of the subsets $\mathcal{H}_{1}$
and $\mathcal{H}_{2}$ are needed in the proof of (\ref{thm:If-,-thenlower}).
\begin{lem}
\label{lem:The-packing-}If $r_{2}\leq\frac{2}{3}\binom{r_{1}}{2}$,
then there exists a subset of $\mathcal{H}_{1}$ with its cardinality
no less than $\exp\left(\frac{r_{2}}{2}\log\frac{\binom{r_{1}}{2}-r_{2}/2}{r_{2}}\right)$
such that the pairwise Hamming distance of the points in this subset
is greater than $r_{2}/2$.
\end{lem}
\begin{proof}
The proof is presented in Appendix \ref{lemma1}.
\end{proof}
\begin{lem}
\label{lem:If-,-then2}If $r_{1}\leq2p_{n}/3$, then there exists
a subset of $\mathcal{H}_{2}$ with its cardinality no less than $\exp\left(\frac{r_{1}}{2}\log\frac{p_{n}-r_{1}/2}{r_{1}}\right)$
such that the pairwise Hamming distance of the points in this subset
is greater than $r_{1}/2$. 
\end{lem}
\begin{proof}
The proof is similar to that of Lemma \ref{lem:The-packing-}.
\end{proof}

\subsection{Proof of (\ref{thm:If-,-thenlower})}
It suffices to prove under $r_{2}\leq\left(r_{1}^{2}-r_{1}\right)/4$.
Since $r_{2}(1+\log(\binom{r_{1}}{2}/r_{2}))\asymp\binom{r_{1}}{2}$
for $\frac{1}{2}\binom{r_{1}}{2}\leq r_{2}\leq\binom{r_{1}}{2}$,
the monotonicity of the minimax risk in the function class reduces
the proof to the case $r_{2}\leq\left(r_{1}^{2}-r_{1}\right)/4$.
Similarly it suffices to prove under $r_{1}\leq p_{n}/2$.

Recall that $B_{0}(r_{1},r_{2};\ddot{\mathbb{R}}_{strong}^{p_{n}})=\left\{ \beta\in\ddot{\mathbb{R}}_{strong}^{p_{n}}:\left\Vert \beta^{(1)}\right\Vert _{0}\le r_{1},\left\Vert \beta^{(2)}\right\Vert _{0}\le r_{2}\right\} $ is the coefficient space of interest and $\mathcal{F}_{0}^{}(r_{1},r_{2};\ddot{\mathbb{R}}_{strong}^{p_{n}})=\left\{ h:h(\mathbf{z})=\mathbf{z}^{T}\beta,\beta\in B_{0}(r_{1},r_{2};\ddot{\mathbb{R}}_{strong}^{p_{n}})\right\} $ is the mean regression function space. For convenience, let $h_{\theta}$, $h_{\vartheta}$ denote the regression functions with coefficents ${\theta},{\vartheta}$ respectively, i.e., $h_{\theta}(\mathbf{z})=\mathbf{z}^{T}\theta$, $h_{\vartheta}(\mathbf{z})=\mathbf{z}^{T}\vartheta$. Let $$B_{0}(r_{1},r_{2};\ddot{\mathbb{R}}_{strong}^{p_{n}})(\epsilon)=\left\{ \beta:\beta\in B_{0}(r_{1},r_{2};\ddot{\mathbb{R}}_{strong}^{p_{n}}),\left\Vert \beta\right\Vert _{2}\leq\epsilon\right\} $$
be an $l_{2}$-ball of radius $\epsilon$ around $0$ in $B_{0}^{\mathcal{}}(r_{1},r_{2};\ddot{\mathbb{R}}_{strong}^{p_{n}})$ and 
$$\mathcal{F}_{0}^{}(r_{1},r_{2};\ddot{\mathbb{R}}_{strong}^{p_{n}})(h,\epsilon_{0})=\left\{ h^\prime:h^\prime(\mathbf{z})=\mathbf{z}^{T}\beta,\beta\in B_{0}(r_{1},r_{2};\ddot{\mathbb{R}}_{strong}^{p_{n}}),d(h^\prime,h)\leq\epsilon_{0}\right\} $$
be the ball of radius $\epsilon_{0}$ around the underlying regression function $h$. Without loss of generality,
we assume $h=0$. The square root of
the empirical $l_{2}$-norm loss $d(h_{\theta},h_{\vartheta}):=\sqrt{\frac{1}{n}\sum_{i=1}^{n}(h_{\theta}(\mathbf{z}_{i})-h_{\vartheta}(\mathbf{z}_{i}))^{2}}=\frac{1}{\sqrt{n}}\left\Vert \mathbf{Z}(\theta-\vartheta)\right\Vert _{2}$ is used to measure the distance between any two functions $h_{\theta},h_{\vartheta}$. We prove the following two cases separately. 

Case 1: $\frac{r_{1}}{2}\log((p_{n}-r_{1}/2)/r_{1})\leq\frac{r_{2}}{2}\log((\binom{r_{1}}{2}-r_{2}/2)/r_{2})$. We consider the subset $\mathcal{H}_{1}^{\prime}=\left\{ \bm{\epsilon}\circ\beta:\beta\in\mathcal{H}_{1}\right\} $
of the $l_{2}$-ball $B_{0}(r_{1},r_{2};\ddot{\mathbb{R}}_{strong}^{p_{n}})(\epsilon)$,
where $\circ$ is the point-wise product of two vectors,
\[
\bm{\epsilon}=\frac{\epsilon}{\sqrt{2}}(\underset{p_{n}}{\underbrace{1/\sqrt{r_{1}},...,1/\sqrt{r_{1}}}},\underset{(p_{n}^{2}-p_{n})/2}{\underbrace{1/\sqrt{r_{2}},...,1/\sqrt{r_{2}}}})
\]
 and 
\[
\mathcal{H}_{1}=\left\{ \beta\in\mathcal{H}:\beta^{(1)}=(\underset{r_{1}}{\underbrace{1,...,1}},\underset{p_{n}-r_{1}}{\underbrace{0,...,0}}),\left\Vert \beta^{(2)}\right\Vert _{0}=r_{2}\right\} .
\]
From Lemma \ref{lem:The-packing-}, there exists a subset $\mathcal{H}_{sub}$
of $\mathcal{H}_{1}$ such that $\left|\mathcal{H}_{sub}\right|\geq\exp(\frac{r_{2}}{2}\log\frac{\binom{r_{1}}{2}-r_{2}/2}{r_{2}})$
and the pairwise Hamming distance of the elements within $\mathcal{H}_{sub}$
is greater than $r_{2}/2$.  Set $\mathcal{H}^{\prime}_{sub}:=\left\{ \bm{\epsilon}\circ\beta:\beta\in\mathcal{H}_{sub}\right\} $. For any $\theta^{\prime},\vartheta^{\prime}\in\mathcal{H}_{sub}^{\prime}$, there exist $\theta,\vartheta\in\mathcal{H}_{sub}$ such that $\left\Vert \theta^{\prime}-\vartheta^{\prime}\right\Vert _{2}=\left\Vert \bm{\epsilon}\circ\theta-\bm{\epsilon}\circ\vartheta\right\Vert _{2}\geq\frac{\epsilon}{\sqrt{2r_{2}}}\sqrt{\rho_{H}(\theta,\vartheta)}\geq\frac{\epsilon}{\sqrt{2r_{2}}}\sqrt{r_{2}/2}=\frac{\epsilon}{2}$. We also have $\left|\mathcal{H}_{sub}^{\prime}\right|=\left|\mathcal{H}_{sub}\right|$ since it is a one-to-one mapping from $\mathcal{H}_{sub}$ to $\mathcal{H}_{sub}^{\prime}$. Thus, we have $\mathcal{H}_{sub}^{\prime}\subseteq B_{0}(r_{1},r_{2};\ddot{\mathbb{R}}_{strong}^{p_{n}})(\epsilon)$ and the pairwise $l_{2}$-distance of the elements in $\mathcal{H}_{sub}^{\prime}$ is greater than $\epsilon/2$.

For any $\theta^{\prime},\vartheta^{\prime}\in \mathcal{H}_{sub}^{\prime}\subseteq B_{0}(r_{1},r_{2};\ddot{\mathbb{R}}_{strong}^{p_{n}})(\epsilon)$, let $h_{\theta^{\prime}},h_{\vartheta^{\prime}}$ be such that $h_{\theta^{\prime}}(\mathbf{z})=\mathbf{z}^{T}\theta^{\prime},h_{\vartheta^{\prime}}(\mathbf{z})=\mathbf{z}^{T}\vartheta^{\prime}$. By SRC assumption with $l_{1}=r_{1}$, $l_{2}=r_{2}$, we have
\[
b_{1}\frac{\epsilon}{2}\leq b_{1}\left\Vert (\theta^{\prime}-\vartheta^{\prime})\right\Vert _{2}\leq d(h_{\theta^{\prime}},h_{\vartheta^{\prime}})
\]

\[
d(h,h_{\vartheta^{\prime}})\leq b_{2}\left\Vert (0-\vartheta^{\prime})\right\Vert _{2}\leq b_2\epsilon.
\]
Let $\epsilon_0=b_2\epsilon$, it follows that $\mathcal{F}_{0}^{}(r_{1},r_{2};\ddot{\mathbb{R}}_{strong}^{p_{n}})(h,\epsilon_{0})$ has a subset $$\mathcal{F}_{sub}:=\left\{ h^\prime:h^\prime(\mathbf{z})=\mathbf{z}^{T}\beta,\beta\in\mathcal{H}_{sub}^{\prime},d(h^\prime,h)\leq\epsilon_{0}\right\},$$ in which the pairwise distance (in terms of $d$) of the functions are no less than $\frac{b_1}{2b_2}\epsilon_0$. This implies that the $\frac{b_{1}}{2b_{2}}$-local $\epsilon_0$-packing
entropy of $\mathcal{F}_{0}^{\mathcal{}}(r_{1},r_{2};\ddot{\mathbb{R}}_{strong}^{p_{n}})(h,\epsilon_0)$
is lower bounded by $\log|\mathcal{F}_{sub}|=\log|\mathcal{H}^{\prime}_{sub}|\geq\frac{r_{2}}{2}\log\frac{\binom{r_{1}}{2}-r_{2}/2}{r_{2}}$. So $\log M_{\mathtt{local}}^{b_{1}/(2b_{2})}(\epsilon_0)$ of $\mathcal{F}_{0}(r_{1},r_{2};\ddot{\mathbb{R}}_{strong}^{p_{n}})$
is no less than $\frac{r_{2}}{2}\log((r_{1}^{2}-r_{1}-r_{2})/2r_{2})$. Then by (7) in \citet{Yang:1999em}, the minimax risk
is lower bounded by 
\[
c_{1}\frac{\sigma^{2}\frac{r_{2}}{2}\log(\frac{r_{1}^{2}-r_{1}-r_{2}}{2r_{2}})}{n}=c_{1}\frac{\sigma^{2}}{n}\left(\frac{r_{1}}{2}\log\frac{p_{n}-r_{1}/2}{r_{1}}\lor\frac{r_{2}}{2}\log\frac{\binom{r_{1}}{2}-r_{2}/2}{r_{2}}\right),
\]
where $c_{1}>0$ is a constant that depends on $b_{1}$ and $b_{2}$
only. 

Case 2: $\frac{r_{1}}{2}\log((p_{n}-r_{1}/2)/r_{1})\geq\frac{r_{2}}{2}\log((\binom{r_{1}}{2}-r_{2}/2)/r_{2})$. We consider the subset $\mathcal{H}^{\prime}_{2}=\epsilon_{1}^{\prime}\mathcal{H}_{2}$
of $B_{0}(r_{1},r_{2};\ddot{\mathbb{R}}_{strong}^{p_{n}})(\epsilon)$,
where $\epsilon_{1}^{\prime}=\epsilon/\sqrt{r_{1}}$ and
\[
\mathcal{H}_{2}:=\left\{ \beta\in\mathcal{H}:\left\Vert \beta^{(1)}\right\Vert _{0}=r_{1},\left\Vert \beta^{(2)}\right\Vert _{0}=0\right\} .
\]
Following the same arguments above, we conclude that the minimax
is lower bounded by 
\[
c_{2}\frac{\sigma^{2}}{n}\frac{r_{1}}{2}\log\frac{p_{n}-r_{1}/2}{r_{1}}=c_{2}\frac{\sigma^{2}}{n}\left(\frac{r_{1}}{2}\log\frac{p_{n}-r_{1}/2}{r_{1}}\lor\frac{r_{2}}{2}\log\frac{\binom{r_{1}}{2}-r_{2}/2}{r_{2}}\right),
\]
where $c_{2}>0$ is a constant that depends on $b_{1}$ and $b_{2}$
only. 

Notice that when $p_{n}/r_{1}\geq2$, we have $\log(p_{n}/r_{1}-\frac{1}{2})\geq\frac{1}{10}(1+\log(p_{n}/r_{1}))$.
Similarly, we have $\log\left(\binom{r_{1}}{2}/r_{2}-\frac{1}{2}\right)\geq\frac{1}{10}\left(1+\log\left(\binom{r_{1}}{2}/r_{2}\right)\right)$
when $\binom{r_{1}}{2}/r_{2}\geq2$. Together with the fact that the lower bounds for the two cases are the same, the minimax risk is lower
bounded by 

\[
c\frac{\sigma^{2}}{n}\left(r_{1}(1+\log(\frac{p_{n}}{r_{1}}))\lor r_{2}(1+\log\frac{\binom{r_{1}}{2}}{r_{2}})\right).
\]
Thus the desired lower bound holds.
\begin{rem}
One way to interpret the imposition of the SRC assumption is that
$\left\Vert \mathbf{Z}\theta-\mathbf{Z}\beta\right\Vert _{2}^{2}$
is indeed up to a constant of the Kullback-Leibler divergence between
two joint densities (the joint distribution of the response variable
$y$ under fixed design) parameterized with $\theta$ and $\beta$
respectively. To see this, let $\mathbf{z}_{i}$ be the $i$-th row
of $\mathbf{Z}$ and we have the joint density $P_{\theta}=(2\pi)^{-n/2}\sigma^{-n}\prod_{i=1}^{n}\exp(-\frac{1}{2}(y_{i}-\mathbf{z}_{i}\theta)^{2}/\sigma^{2})$
with parameter $\theta$. The K-L distance is then $D(P_{\theta}||P_{\beta})=\frac{1}{2\sigma^{2}}\sum_{i=1}^{n}(\mathbf{z}_{i}\beta-\mathbf{z}_{i}\theta)^{2}=\frac{1}{2\sigma^{2}}\left\Vert \mathbf{Z}\theta-\mathbf{Z}\beta\right\Vert _{2}^{2}$.
\end{rem}
\subsection {Proof of Lemma \ref{lem:The-packing-}}\label{lemma1}

First we have $\left|\mathcal{H}_{1}\right|=\binom{\left(r_{1}^{2}-r_{1}\right)/2}{r_{2}}2^{r_{2}}$
since the main effects are fixed. Fix $z\in\mathcal{H}_{1}$, let
$\mathcal{A}$ denote the collection of all the points in $\mathcal{H}_{1}$
that are within $\frac{r_{2}}{2}$ Hamming distances to $z$, i.e.,
$\mathcal{A}=\left\{ z^{\prime}\in\mathcal{H}_{1}:\rho_{H}(z,z^{\prime})\leq r_{2}/2\right\} $.
It follows that the cardinality of $\mathcal{A}$ is bounded above:
\begin{eqnarray*}
\left|\mathcal{A}\right| & \leq & \binom{\binom{r_{1}}{2}}{r_{2}/2}3^{r_{2}/2}.
\end{eqnarray*}
For this upper bound, since the main effects are fixed for any point
in $\mathcal{H}_{1}$, we only need to pick $r_{2}/2$ positions of the interaction effects where
$z^{\prime}$ is different from $z$. In the remaining
interaction effect positions, $z^{\prime}$ is the same as $z$. It gives us at most $\binom{\binom{r_{1}}{2}}{r_{2}/2}$
possible choices of the $r_{2}/2$ positions out of the $\binom{r_1}{2}$ coordinates. For these $r_{2}/2$ positions, $z^{\prime}$ can
take any values in $\{-1,1,0\}$, thus the desired upper bound follows.

Let $\mathcal{B}$ be a subset of $\mathcal{H}_{1}$ such that
$\left|\mathcal{B}\right|\leq m:=\binom{\binom{r_{1}}{2}}{r_{2}}/\binom{\binom{r_{1}}{2}}{r_{2}/2}$.
Consider the collection of the points in $\mathcal{H}_{1}$
that are within $r_{2}/2$ Hamming distance to some element in $\mathcal{B}$, i.e., $\{ z\in\mathcal{H}_{1}:\rho_{H}(z,z^{\prime})\leq\frac{r_{2}}{2}\ for\ some\ z^{\prime}\in\mathcal{B}\}$. We have
\begin{eqnarray*}
 &  & \left|\left\{ z\in\mathcal{H}_{1}:\rho_{H}(z,z^{\prime})\leq\frac{r_{2}}{2}\ for\ some\ z^{\prime}\in\mathcal{B}\right\} \right|\\
 & \leq & \left|\mathcal{B}\right|\left|\mathcal{A}\right|\\
 & \leq & \frac{\binom{\binom{r_{1}}{2}}{r_{2}}}{\binom{\binom{r_{1}}{2}}{r_{2}/2}}\cdot\binom{\binom{r_1}{2}}{r_{2}/2}3^{r_{2}/2}\\
 & < & \binom{\binom{r_{1}}{2}}{r_{2}}2^{r_{2}}\\
 & = & \left|\mathcal{H}_{1}\right|.
\end{eqnarray*}
The strictly less inequality implies that for any set $\mathcal{B}\subset\mathcal{H}_{1}$
with $\left|\mathcal{B}\right|\leq m$, $\exists z\in\mathcal{H}_{1}$
such that $\rho_{H}(z,z^{\prime})>\frac{1}{2}r_{2}$ for all $z^{\prime}\in\mathcal{B}$.
By induction, we can create a set $\mathcal{B}\subset\mathcal{H}_{1}$
with $\left|\mathcal{B}\right|>m$ such that Hamming distance between
any two elements in $\mathcal{B}$ exceeds $\frac{1}{2}r_{2}$. Next, we introduce one
useful inequality. When $0\leq B\leq\frac{2}{3}A$ for $A,B\in\mathbb{N}$,
we have
\[
\frac{\binom{A}{B}}{\binom{A}{\frac{B}{2}}}=\frac{(A-\frac{B}{2})!(\frac{B}{2})!}{(A-B)!(B)!}=\prod_{j=1}^{B/2}\frac{A-B+j}{\frac{B}{2}+j}\geq\prod_{j=1}^{B/2}\frac{A-B+\frac{B}{2}}{\frac{B}{2}+\frac{B}{2}}=(\frac{A-\frac{B}{2}}{B})^{B\text{/2}}.
\]
When $r_{2}\leq(r_{1}^{2}-r_{1})/3$, we have
\begin{eqnarray*}
m & = & \frac{\binom{\binom{r_{1}}{2}}{r_{2}}}{\binom{\binom{r_{1}}{2}}{r_{2}/2}}\ge\left(\frac{\binom{r_{1}}{2}-r_{2}/2}{r_{2}}\right)^{r_{2}/2}.
\end{eqnarray*}
Thus, 
\begin{eqnarray*}
\log m & \ge & \frac{r_{2}}{2}\log\frac{\binom{r_{1}}{2}-\frac{r_{2}}{2}}{r_{2}}.
\end{eqnarray*}
The desired result follows.

\section {Proof of Theorem 4.1}
\begin{proof}
The proofs are similar to the arguments for strong heredity with
slight differences. 

To prove the upper bound under weak heredity, we instead consider the model $\hat{\mathtt{I}}=\arg\min_{\mathtt{I}\in\mathtt{I}^{weak}_{r_1,r_2}}\sum_{i=1}^{n}(Y_{i}-\hat{Y}_{i}^{\mathtt{I}})^2$ that minimizes the residual sum of squares over all the models that have $r_1$ non-zero main effects and $r_2$ non-zero interaction effects under weak heredity.
The model descriptive complexity is thus different from the strong
heredity. In this case, $C_{\mathtt{I}_{r_{1},r_{2}}^{weak}}=\log\binom{p_{n}}{r_{1}}+\log\binom{K}{r_{2}}$
with $K=r_{1}(p_{n}-(r_{1}+1)/2)$ for $1\leq r_{1}\leq p_{n}\land n\textrm{ and }0\leq r_{2}\leq(r_{1}p_{n}-\binom{r_{1}}{2}-r_{1})\land n$. The ABC criteria for the models
are defined as in (\ref{eq:1}). The same
arguments in the proof of (\ref{thm:When-,-we1}) can then be used. 

To prove the lower bound under weak heredity, we consider the set
\[
\mathcal{H}_{weak}=\left\{ \beta\in\ddot{\mathbb{R}}_{weak}^{p_{n}}:\beta\in\{-1,0,1\}^{p_{n}+\binom{p_{n}}{2}},\left\Vert \beta^{(1)}\right\Vert _{0}\leq r_{1},\left\Vert \beta^{(2)}\right\Vert _{0}\leq r_{2}\right\} .
\]
Then the two important subsets are instead
\[
\mathcal{H}_{1}=\left\{ \beta\in\mathcal{H}_{weak}:\beta^{(1)}=(\underset{r_{1}}{\underbrace{1,...,1}},\underset{p_{n}-r_{1}}{\underbrace{0,...,0}}),\left\Vert \beta^{(2)}\right\Vert _{0}=r_{2}\right\} 
\]
 and
\[
\mathcal{H}_{2}=\left\{ \beta\in\mathcal{H}_{weak}:\left\Vert \beta^{(1)}\right\Vert _{0}=r_{1},\left\Vert \beta^{(2)}\right\Vert _{0}=0\right\} .
\]

Similar metric entropy results of the above two subsets can be derived
in the same fashion as in Lemmas \ref{lem:The-packing-} and \ref{lem:If-,-then2}. Other arguments are
the same as in the proof of (\ref{thm:If-,-thenlower}). 
\end{proof}
\section {Proof of Theorem 4.2}
\begin{proof}
For the upper bound under no heredity, we consider the model
$\hat{\mathtt{I}}=\arg\min_{\mathtt{I}\in\mathtt{I}^{no}_{r_1,r_2}}\sum_{i=1}^{n}(Y_{i}-\hat{Y}_{i}^{\mathtt{I}})^{2}$
with the model descriptive complexity $C_{\mathtt{I}_{r_{1},r_{2}}^{no}}=\log\binom{p_{n}}{r_{1}}+\log\binom{\binom{p_{n}}{2}}{r_{2}}$ for $1\leq r_{1}\leq p_{n}\land n\textrm{ and }0\leq r_{2}\leq\binom{p_{n}}{2}\land n$. The ABC criteria for the models
are defined as in (\ref{eq:1}).

For the lower bound under no heredity, we consider the set
\[
\mathcal{H}_{no}=\left\{ \beta\in\ddot{\mathbb{R}}^{p_{n}}:\beta\in\{-1,0,1\}^{p_{n}+\binom{p_{n}}{2}},\left\Vert \beta^{(1)}\right\Vert _{0}\leq r_{1},\left\Vert \beta^{(2)}\right\Vert _{0}\leq r_{2}\right\} .
\]
Then the two important subsets are instead
\[
\mathcal{H}_{1}=\left\{ \beta\in\mathcal{H}_{no}:\beta^{(1)}=(\underset{r_{1}}{\underbrace{1,...,1}},\underset{p_{n}-r_{1}}{\underbrace{0,...,0}}),\left\Vert \beta^{(2)}\right\Vert _{0}=r_{2}\right\} 
\]
 and
\[
\mathcal{H}_{2}=\left\{ \beta\in\mathcal{H}_{no}:\left\Vert \beta^{(1)}\right\Vert _{0}=r_{1},\left\Vert \beta^{(2)}\right\Vert _{0}=0\right\} .
\]
Similar metric entropy results of the above two subsets can be derived
in the same fashion as Lemmas \ref{lem:The-packing-} and \ref{lem:If-,-then2}. 

Other arguments are the same as in the proofs of (\ref{thm:When-,-we1}) and (\ref{thm:If-,-thenlower}). 
\end{proof}

\section{Proof of Theorem 7.1}
The model descriptive complexity term $\lambda\sigma^{2}C_{\mathtt{I}}$
plays a fundamental role in model selection theory \citep{barron1991minimum,barron1999risk,yang1999model,wang2014adaptive}. Since we are considering models with interaction terms, the
model descriptive complexity $C_{\mathtt{I}}$ reflects our comprehension
of the model complexity other than the total number of parameters
only. The detailed designation of the descriptive complexity usually
depends on the class of models of interest. Instead
of interpreting $C_{\mathtt{I}}$ as the code length (or description
length) of describing the model index, one can also treat $\exp(-C_{\mathtt{I}})$
as the prior probability assigned to the model from a Bayesian
viewpoint.
\begin{proof}
The candidate set can be represented as the union of the candidate
sets under three heredity conditions, i.e., $\bar{\mathcal{F}}=\mathcal{F}_{strong}\cup\mathcal{F}_{weak}\cup\mathcal{F}_{no}$, with $$\mathcal{F}_{strong}:= \{\mathtt{I}_{p_{n},(p_{n}^{2}-p_{n})/2}\}\cup\{\mathtt{I}_{k_{1},k_{2}}^{strong}\},$$ $$\mathcal{F}_{weak}:= \{\mathtt{I}_{p_{n},(p_{n}^{2}-p_{n})/2}\}\cup\{\mathtt{I}_{k_{1},k_{2}}^{weak}\},$$ $$\mathcal{F}_{no}:= \{\mathtt{I}_{p_{n},(p_{n}^{2}-p_{n})/2}\}\cup\{\mathtt{I}_{k_{1},k_{2}}^{no}\}.$$
When $\ensuremath{h\in\mathcal{F}_{0}(r_{1},r_{2};\ddot{\mathbb{R}}_{strong}^{p_{n}})}$, there exists a specific model in $\mathcal{F}_{strong}$ such that the projection estimator of this model is equal to $\mathbf{h}$. Also, the projection of $\mathbf{h}$ onto the full design matrix is still $\mathbf{h}$. We denote the two models as $\mathtt{I}_{r_{1},r_{2}}$ and $\mathtt{I}_{p_{n},(p_{n})(p_n-1)/2}$ respectively.
It follows that
\begin{eqnarray}
E(\mathcal{L}(\hat{Y}^{_{\bar{\mathcal{F}}}})) & \leq & c\underset{\mathtt{I}\in\bar{\mathcal{F}}}{\inf}\left(\frac{1}{n}\left\Vert \mathbf{h}_{\mathtt{I}}-\mathbf{h}\right\Vert _{2}^{2}+\frac{\sigma^{2}r_{\mathtt{I}}}{n}+\frac{\lambda\sigma^{2}C_{\mathtt{I}}}{n}\right)\nonumber\\
 & \leq & c\underset{\mathtt{I}\in\mathcal{F}_{strong}}{\inf}\left(\frac{1}{n}\left\Vert \mathbf{h}_{\mathtt{I}}-\mathbf{h}\right\Vert _{2}^{2}+\frac{\sigma^{2}r_{\mathtt{I}}}{n}+\frac{\lambda\sigma^{2}C_{\mathtt{I}}}{n}\right)\label{eq:123}\\
 & \leq & c\left(\left\Vert \mathbf{h}_{\mathtt{I}_{r_{1},r_{2}}}-\mathbf{h}\right\Vert ^{2}+\frac{\sigma^{2}r_{\mathtt{I}_{r_{1},r_{2}}}}{n}+\frac{\lambda\sigma^{2}C_{\mathtt{I}_{r_{1},r_{2}}}}{n}\right)\nonumber\\
 &  & \land\ c\left(\left\Vert \mathbf{h}_{\mathtt{I}_{p_{n},p_{n}(p_{n}-1)/2}}-\mathbf{h}\right\Vert ^{2}+\frac{\sigma^{2}R_{\mathbf{Z}}}{n}+\frac{-\lambda\sigma^{2}\log \pi_0}{n}\right)\nonumber\\
 & = & \underset{(i)}{\underbrace{\frac{c}{n}\left(\sigma^{2}r_{\mathtt{I}_{r_{1},r_{2}}}+\lambda\sigma^{2}C_{\mathtt{I}_{r_{1},r_{2}}}\right)}}\land\underset{(ii)}{\underbrace{\frac{c}{n}\left(\sigma^{2}R_{\mathbf{Z}}-\lambda\sigma^{2}\log \pi_0\right)}},\label{eq:1234}
\end{eqnarray}
where $R_{\mathbf{Z}}$ is the rank of the full design matrix, the first inequality follows from (\ref{eq:yang}), the second inequality follows from $\mathcal{F}_{strong}\subseteq\bar{\mathcal{F}}$ and the third inequality results from the evaluation
of (\ref{eq:123}) at $\mathtt{I}_{r_{1},r_{2}}$ and $\mathtt{I}_{p_{n},p_{n}(p_n-1)/2}$. The
two terms $(i)$ and $(ii)$ are bounded as follows:
\sloppy
\begin{eqnarray*}
(i) & \leq & \frac{c_{1}\lambda}{n}\sigma^{2}\left(\frac{r_{1}+r_{2}}{\lambda}-\log \pi_1+\log p_{n}+\log{r_{1} \choose 2}+\log\binom{p_{n}}{r_{1}}+\log\binom{\binom{r_{1}}{2}}{r_{2}}\right)\\
 & \leq & \frac{c_{1}\lambda}{n}\sigma^{2}\left(\frac{r_{1}+r_{2}}{\lambda}-\log \pi_1+r_{1}(1+\log\frac{p_{n}}{r_{1}})\right.\\
 &  & \left.+\log r_{1}^{2}+r_{1}(1+\log\frac{p_{n}}{r_{1}})+r_{2}(1+\log\frac{\binom{r_{1}}{2}}{r_{2}})\right)\\
 & \leq & \frac{c_{2}}{n}\sigma^{2}\left(r_{1}(1+\log\frac{p_{n}}{r_{1}})+r_{2}(1+\log\frac{\binom{r_{1}}{2}}{r_{2}})\right),
\end{eqnarray*}
and
\begin{eqnarray*}
(ii) & \leq & \frac{c}{n}\left(\sigma^{2}R_{\mathbf{Z}}-\lambda\sigma^{2}\log \pi_0\right)\\
 & \leq & \frac{c_{3}}{n}\sigma^{2}R_{\mathbf{Z}}.
\end{eqnarray*}
Therefore, we have 
\[
E(\mathcal{L}(\hat{Y}^{_{\bar{\mathcal{F}}}}))\leq\frac{\max(c_{2},c_{3})\cdot\sigma^{2}}{n}\left[\left(r_{1}(1+\log\frac{p_{n}}{r_{1}})+r_{2}(1+\log\frac{\binom{r_{1}}{2}}{r_{2}})\right)\land R_{\mathbf{Z}}\right],
\]
where $c_{1}$, $c_{2}$, $c_{3}$ are some constants that depend
only on the constant $\lambda$. Thus the desired minimax upper bounded follows.

When $\ensuremath{h\in\mathcal{F}_{0}(r_{1},r_{2};\ddot{\mathbb{R}}_{weak}^{p_{n}})}$
or $\ensuremath{h\in\mathcal{F}_{0}(r_{1},r_{2};\ddot{\mathbb{R}}^{p_{n}})}$, with $\mathtt{I}\in\mathcal{F}_{weak}$ or $\mathtt{I}\in\mathcal{F}_{no}$ replacing $\mathtt{I}\in\mathcal{F}_{strong}$ in (\ref{eq:123}), the quantity $(i)$ in (\ref{eq:1234}) will instead be no greater than
$$\frac{c_{1}\lambda}{n}\sigma^{2}\left(\frac{r_{1}+r_{2}}{\lambda}-\log \pi_2+\log p_{n}+\log{K}+\log\binom{p_{n}}{r_{1}}+\log\binom{K}{r_{2}}\right)$$ 
with $K=r_{1}p_{n}-\binom{r_{1}}{2}-r_{1}$ under weak heredity $\ensuremath{h\in\mathcal{F}_{0}(r_{1},r_{2};\ddot{\mathbb{R}}_{weak}^{p_{n}})}$, or 
$$(i) \leq \frac{c_{1}\lambda}{n}\sigma^{2}\left(\frac{r_{1}+r_{2}}{\lambda}-\log \pi_3+\log p_{n}+\log{p_{n} \choose 2}+\log\binom{p_{n}}{r_{1}}+\log\binom{\binom{p_{n}}{2}}{r_{2}}\right)$$ 
under no heredity $\ensuremath{h\in\mathcal{F}_{0}(r_{1},r_{2};\ddot{\mathbb{R}}^{p_{n}})}$. The different constants $\pi_2, \pi_3$ does not affect the conclusion in terms of order. Following the same arguments in the proof of strong heredity, the desired results follow when the underlying heredity condition is weak heredity or no heredity.
\end{proof}

\section{An Example When SRC is
not Satisfied}\label{subsec:An-example-when}

For simplicity, let us consider an example where the regression mean function includes only one main effect term, i.e., $r_{1}=1, r_2=0$. The corresponding SRC assumption with $l_{1}=r_{1}=1, l_2=r_2=0$ will be that there exist
constants $b_{1},b_{2}>0$ (not depend on $n$) such that for any
$\beta\in\mathbb{R}^{p_n}$ with $\left\Vert \beta\right\Vert _{0}\le2$, we have
\begin{equation}
b_{1}\left\Vert \beta\right\Vert _{2}\leq\frac{1}{\sqrt{n}}\left\Vert \mathbf{Z}\beta\right\Vert _{2}\leq b_{2}\left\Vert \beta\right\Vert _{2}\label{eq:eg},
\end{equation}
where the design matrix $\mathbf{Z}=\mathbf{X}$ is the matrix that contains the main effects.

Assume the first $R_{\mathbf{Z}}$ columns of $\mathbf{Z}$ are linearly
independent and denote $\mathbf{Z}=(\mathbf{Z}^{1},\mathbf{Z}^{2})$, where $\mathbf{Z}^{1}=(\mathbf{Z}_{1},...,\mathbf{Z}_{R_{\mathbf{Z}}})$ is the $n\times R_{\mathbf{Z}}$ submatrix with $rank(\mathbf{Z}^{1})=R_{\mathbf{Z}}$.
Suppose the submatrix $\mathbf{Z}^{1}$ satisfies
the SRC assumption. Assume that $\left\Vert \mathbf{Z}_{i}\right\Vert _{2}=f(n)$
for $1\leq i\leq p_{n}$. For the purpose of illustration, we set $f(n)=\sqrt{n}$.

Let $A$ be the collection of all columns in $\mathbf{Z}^{2}$: $A=\{z|z=\mathbf{Z}^{1}\alpha,\alpha\in\mathbb{R}^{R_{\mathbf{Z}}},\left\Vert z\right\Vert _{2}=f(n)\}$.
Then $A$ should satisfy that $\forall z,z^{\prime}\in A$, we have
$b_{1}\leq\frac{1}{\sqrt{n}}\left\Vert a_{1}z+a_{2}z^{\prime}\right\Vert _{2}\leq b_{2}$
for all $a_{1},a_{2}\in\mathbb{R}$ and $a_{1}^{2}+a_{2}^{2}=1$.
We know 
\[
\frac{1}{\sqrt{n}}\left\Vert a_{1}z+a_{2}z^{\prime}\right\Vert _{2}=\frac{1}{\sqrt{n}}\sqrt{a_{1}^{2}\left\Vert z\right\Vert _{2}^{2}+a_{2}^{2}\left\Vert z^{\prime}\right\Vert _{2}^{2}+2a_{1}a_{2}\left\Vert z\right\Vert _{2}\left\Vert z^{\prime}\right\Vert _{2}\cos\theta},
\]
where $\theta$ is the angle between two $n$-dimensional vectors
$z$ and $z^{\prime}$. 

Thus we have 
\begin{eqnarray*}
 &  & \frac{1}{\sqrt{n}}\sqrt{a_{1}^{2}\left\Vert z\right\Vert _{2}^{2}+a_{2}^{2}\left\Vert z^{\prime}\right\Vert _{2}^{2}+2a_{1}a_{2}\left\Vert z\right\Vert _{2}\left\Vert z^{\prime}\right\Vert _{2}\cos\theta}\\
 & = & \frac{f(n)}{\sqrt{n}}\sqrt{a_{1}^{2}+a_{2}^{2}+2a_{1}a_{2}\cos\theta}\\
 & = & \sqrt{1+2a_{1}a_{2}\cos\theta}.
\end{eqnarray*}
Then $\sqrt{1+2a_{1}a_{2}\cos\theta}\geq b_{1}$ for all $a_{1}^{2}+a_{2}^{2}=1$
(otherwise $\frac{1}{\sqrt{n}}\left\Vert a_{1}z+a_{2}\mathbf{Z}_{i}\right\Vert _{2}$
is less than $b_{1}$, which violates the SRC assumption). Since $-1\leq2a_{1}a_{2}\leq1$
for $a_{1}^{2}+a_{2}^{2}=1$, we have $b_{1}\leq\sqrt{1-\left|\cos\theta\right|}$,
which implies $\left|\cos\theta\right|\leq1-b_{1}^{2}$. That means
the pairwise $l_{2}$ distance between any two elements in $A$ should
be greater than $\sqrt{2}b_{1}$ and less than $\sqrt{4-2b_{1}^{2}}$.
It is well known that the $\epsilon$-covering entropy of the $R_{\mathbf{Z}}$-dimensional
unit ball $\mathbb{B}$ is of order $R_{\mathbf{Z}}\log(1/\epsilon)$. We denote $\sqrt{n}\mathbb{B}$ as a ball of radius $\sqrt{n}$.
Let $\epsilon=\sqrt{2}b_{1}/2$, there exists a positive constant $c_1$ such that $\log N(\epsilon;\sqrt{n}\mathbb{B},l_{2})\leq c_1R_{\mathbf{Z}}\log(\sqrt{n}/\epsilon)=c_1R_{\mathbf{Z}}\log(\sqrt{2n}/b_1)$.
Since $A$ is a $2\epsilon$-packing set of a ball of radius $f(n)=\sqrt{n}$,
its cardinality satisfies $\log|A|\leq\log M(2\epsilon;\sqrt{n}\mathbb{B},l_{2})$.  The covering number and the packing number are closely related as
in the well-known inequality $M(\epsilon;\mathcal{X},\rho)\leq N(\frac{\epsilon}{2};\mathcal{X},\rho)\leq M(\frac{\epsilon}{2};\mathcal{X},\rho)$. Thus we have $\log|A|\leq\log M(2\epsilon;\sqrt{n}\mathbb{B},l_{2})\leq\log N(\epsilon;\sqrt{n}\mathbb{B},l_{2})\leq c_1R_{\mathbf{Z}}\log(\sqrt{2n}/b_1)$, which implies $A$ has at most $(\sqrt{2n}/b_{1})^{c_1R_{\mathbf{Z}}}$ elements
under the SRC assumption. Thus, as long as $p_{n}>(\sqrt{2n}/b_{1})^{c_1R_{\mathbf{Z}}}$, the SRC assumption will not be
satisfied because the SRC assumption requires that (\ref{eq:eg}) must hold for any pair of columns in $\mathbf{Z}$. In this case, the lower bound $r_{1}(1+\log(p_{n}/r_{1}))/n$ in our theorems does not apply.

\bibliographystyle{agsm}
\bibliography{ims-template}

\end{document}